\documentclass[12pt]{article}
\usepackage{amssymb,amsmath,amsopn,amsfonts}
\usepackage{latexsym}
\usepackage{array}
\usepackage{graphicx}
\usepackage{amsthm}
\usepackage{colordvi,multicol}
\usepackage{epsf}
\usepackage{tikz}
\usepackage[T1]{fontenc}
\usepackage[utf8]{inputenc}
\usepackage{authblk}
\usepackage{hyperref}
\newtheorem{lemma}{Lemma}
\newtheorem{Thm}{Theorem}
\newtheorem{remark}{Remark}
\newtheorem{definition}{Definition}
\newtheorem{hyp}{Hypothesis}

\newfam\msbfam
\font\tenmsb=msbm10 \textfont\msbfam=\tenmsb \font\sevenmsb=msbm7
\scriptfont\msbfam=\sevenmsb \font\fivemsb=msbm5
\scriptscriptfont\msbfam=\fivemsb

\numberwithin{equation}{section}
\numberwithin{lemma}{section}
\numberwithin{Thm}{section}
\numberwithin{remark}{section}
\numberwithin{hyp}{section}
\numberwithin{definition}{section}

\providecommand{\keywords}[1]{\textbf{\textbf{Keywords:}} #1}

\begin{document}
\bigbreak

\title{\bf Large deviation principle for the two-dimensional stochastic Navier-Stokes equations with anisotropic viscosity \thanks{Research supported in part  by NSFC (No.11771037) and Key Lab of Random Complex Structures and Data Science, Chinese Academy of Science. Financial supported by the DFG through the CRC 1283 "Taming uncertainty and profiting from randomness and low regularity in analysis, stochastics and their applications" is acknowledged.}}
\author[1,2]{\bf Bingguang Chen\thanks{bchen@math.uni-bielefeld.de}}
\author[1]{\bf Xiangchan Zhu\thanks{Corresponding author: zhuxiangchan@126.com}}
\date{}
\affil[1]{Academy of Mathematics and System Science, Chinese Academy of Science, Beijing 100190, China }
\affil[2]{Department of Mathematics, University of Bielefeld, D-33615 Bielefeld, Germany}
\renewcommand*{\Affilfont}{\small\it}
\renewcommand\Authand{, }

\maketitle

\begin{abstract}
In this paper we establish the large deviation principle for the the two-dimensional stochastic Navier-Stokes equations with anisotropic viscosity both for small noise and for short time. The proof for  large deviation principle is based on the weak convergence approach. For small time asymptotics we use the exponential equivalence to prove the result. 
\end{abstract}

\vspace{2mm}

\section{Introduction}
The main aim of this work is to establish large deviation principle and small time asymptotics  for the stochastic Navier-Stokes equation with anisotropic viscosity. We consider the following stochastic Navier-Stokes equation with anisotropic viscosity on the two dimensional (2D) torus $\mathbb{T}^2=\mathbb{R}^2/(2\pi\mathbb{Z})^2$:
\begin{equation}\label{equation before projection}\aligned
&du=\partial^2_1 udt-u\cdot \nabla u dt+\sigma(t, u)dW(t)-\nabla pdt,\\
&\text{div } u=0,\\
&u(0)=u_0,
\endaligned
\end{equation}
where $u(t,x)$ denotes the velocity field at time $t\in[0,T]$ and position $x\in\mathbb{T}^2$, $p$ denotes the pressure field, $\sigma$ is the random external force and $W$ is an $l^2$-cylindrical Wiener process. 

 Let's first recall the classical Navier-Stokes (N-S) equation which is given by
\begin{equation}\label{equation classical}\aligned
&du=\nu\Delta udt-u\cdot \nabla u dt-\nabla pdt,\\
&\text{div } u=0,\\
&u(0)=u_0,
\endaligned
\end{equation}
where $\nu>0$ is the viscosity of the fluid. (\ref{equation classical}) describes the time evolution of an incompressible fluid. In 1934, J. Leray proved global existence of finite energy weak solutions for the deterministic case in the whole space $\mathbb{R}^d$ for $d=2, 3$ in the seminar paper \cite{L33}. For more results on deterministic N-S equation, we refer to \cite{CKN86}, \cite{Te79}, \cite{Te95}, \cite{KT01} and reference therein. For the stochastic case, there exists a great amount of literature too. The existence and uniqueness of solutions and ergodicity property to the stochatic 2D Navier-Stokes equation have been obtained (see e.g. \cite{FG95}, 
\cite{MR05}, \cite{HM06}). Large deviation principles for the two-dimensional stochastic N-S equations  have been established in \cite{CM10} and \cite{SS06}.

Compared to \eqref{equation classical}, \eqref{equation before projection} only has partial dissipation, which can be viewed as an intermediate equation between N-S equation and Euler equation.  System of this type appear in in geophysical fluids (see for instance \cite{CDGG06} and \cite{Ped79}). Instead of putting the classical viscosity $-\nu\Delta$ in (\ref{equation classical}), meteorologist often modelize turbulent diffusion by putting a viscosity of the form: $-\nu_h\Delta_h-\nu_3\partial^2_{x_3}$, where $\nu_h$ and $\nu_3$ are empiric constants, and $\nu_3$ is usually much smaller than $\nu_h$. We refer to the book of J. Pedlovsky \cite[Chapter 4]{Ped79} for a more complete discussion. For the 3 dimensional case there is no result concerning global existence of weak solutions. 

In the 2D case, \cite{LZZ18} investigates both the deterministic system  and the stochastic system (\ref{equation before projection}) for $H^{0,1}$ initial value (For the definition of space see Section 2). The main difference in obtaining the global well-posedness for \eqref{equation before projection} is that the $L^2$-norm estimate is not enough to establish $L^2([0,T], L^2)$ strong convergence due to lack of compactness. In \cite{LZZ18}, the proof is based on an additional $H^{0,1}$-norm estimate.  In this paper, we want to establish the large deviation principles for the two-dimensional stochastic Navier-Stokes equations with anisotropic viscosity both for small noise and for short time.  

The large deviation theory concerns the asymptotic behavior of a family of random variables $X_\varepsilon$ and we refer to the monographs \cite{DZ92} and \cite{Str84} for many historical remarks and extensive references. It asserts that for some tail or extreme event $A$, $P(X_\varepsilon \in A)$ converges to zero exponentially fast as $\varepsilon\rightarrow 0$ and the exact rate of convergence is given by the so-called rate function. The large deviation principle was first established by Varadhan in \cite{V66} and he also studied the small time asymptotics of finite dimensional diffusion processes in \cite{Var67}. Since then, many important results concerning the large deviation principle have been established. For results on the large deviation principle for stochastic differential equations in finite dimensional case we refer to \cite{FW84}. For the extensions to infinite dimensional diffusions or SPDE, we refer the readers to \cite{BDM08},  \cite{CM10},  \cite{DM09}, \cite{Liu09}, \cite{LRZ13}, \cite{RZ08}, \cite{XZ08}, \cite{Zh00} and the references therein. 

We first study  the small noise large deviations by using the weak convergence approach. This approach is mainly based on a variational representation formula for certain functionals of infinite dimensional Brownian Motion, which was established by Budhiraja and Dupuis in \cite{BD00}. The main advantage of the weak convergence approach is that one can avoid some exponential probability estimates, which might be very difficult to derive for many infinite dimensional models. To use the weak convergence approach, we need to prove two conditions in Hypothesis \ref{Hyp}. In \cite{Liu09} and \cite{LRZ13}, the authors use integration by parts and lead to some extra condition on diffusion coefficient. In \cite{CM10}, the authors use time discretization and require time-regularity of diffusion coefficient.  In this paper, we use the argument  in \cite{WZZ}, in which the authors  prove a moderate deviation principle by this argument, i.e.  we first establish the convergence in  $L^2([0,T],L^2)$  and then by using this and It\^o's formula,  $L^\infty([0,T], L^2)\bigcap L^2([0,T], H^{1,0})$ convergence can be obtained.  By this argument, we can drop the extra condition on diffusion coefficient in \cite{Liu09} and \cite{CM10}.

For the small time asymptotics (large deviations) of the two-dimensional stochastic Navier-Stokes equations with anisotropic viscosity.  This describes the limiting behaviour of the solution in time interval $[0,t]$ as $t$ goes to zero. Another motivation will be to get the following Varadhan identity through the small time asymptotics:
$$\lim_{t\rightarrow 0}2t\log P(u(0)\in B, u(t)\in C)=-d^2(B,C),$$
where $d$ is an appropriate Riemannian distance associated with the diffusion generated by the solutions of the two-dimensional stochastic Navier-Stokes equations with anisotropic viscosity.   The small time asymptotics  is also theoretically interesting, since the study involves the investigation of the small  noise and the effect of the small, but highly  nonlinear drift. 

To prove the small time asymptotics, we follow the idea of \cite{XZ08} to prove the solution to (\ref{equation before projection}) is exponentially equivalent to the solution to the linear equation. The main difference compared to \cite{XZ08} is that similar to \cite{LZZ18} $L^2$-norm estimate is not enough due to less dissipation and we have to do $H^{0,1}$-norm estimate.

 \vskip.10in

{\textbf{Organization of the paper}}

In Section 2, we introduce the basic  notation, definition and recall some preliminary results.  In Section 3, we will build the small noise large deviation principle. In Section 4, we prove the small time asymptotics for the the two-dimensional stochastic Navier-Stokes equations with anisotropic viscosity.

 \vskip.10in

{\textbf{Acknowledgement}}

The authors would like to thank Rongchan Zhu and Siyu Liang for helpful discussions.

\section{Preliminary}

 \vskip.10in
{\textbf{Function spaces on $\mathbb{T}^2$}}

We first recall some definitions of function spaces for the two dimensional torus $\mathbb{T}^2$. 

Let $\mathbb{T}^2=\mathbb{R}/2\pi\mathbb{Z}\times\mathbb{R}/2\pi\mathbb{Z}=(\mathbb{T}_h, \mathbb{T}_v)$ where $h$ stands for the horizonal variable $x_1$ and $v$ stands for the vertical variable $x_2$. For exponents $p,q\in [1, \infty)$, we denote the space $L^p(\mathbb{T}_h, L^q(\mathbb{T}_v))$ by $L^p_h(L^q_v)$, which is endowed with the norm
$$\|u\|_{L^p_h(L^q_v)(\mathbb{T}^2)}:=\{\int_{\mathbb{T}_h}(\int_{\mathbb{T}_v}|u(x_1, x_2)|^qdx_2)^{\frac{p}{q}}dx_1\}^{\frac{1}{p}}.$$

Similar notation for $L^p_v(L^q_h)$. In the case $p, q=\infty$, we denote $L^\infty$ the essential supremum norm. 
Throughout the paper, we denote various positive constants by the same letter $C$.

For $u\in L^2(\mathbb{T}^2)$, we consider the Fourier expansion of $u$:
$$u(x)=\sum_{k\in\mathbb{Z}^2}\hat{u}_ke^{ik\cdot x} \text{ } \text{with} \text{ } \hat{u}_k=\overline{\hat{u}_{-k}},$$
where $\hat{u}_k:=\frac{1}{(2\pi)^2}\int_{[0, 2\pi]\times[0, 2\pi]}u(x)e^{-ik\cdot x}dx$ denotes the Fourier coefficient of $u$ on $\mathbb{T}^2$.

Define the Sobolev norm:
$$\|u\|_{H^s}^2:=\sum_{k\in\mathbb{Z}^2}(1+|k|^2)^s |\hat{u}_k|^2,$$
and the anisotropic Sobolev norm:
$$\|u\|^2_{H^{s, s'}}=\sum_{k\in\mathbb{Z}^2}(1+|k_1|^2)^s(1+|k_2|^2)^{s'}|\hat{u}_k|^2,$$
where $k=(k_1, k_2)$. We define the Sobolev spaces $H^s(\mathbb{T}^2)$, $H^{s,s'}(\mathbb{T}^2)$ as the completion of $C^\infty(\mathbb{T}^2)$ with the norms $\|\cdot\|_{H^s}$, $\|\cdot\|_{H^{s,s'}}$ respectively. 

To formulate the stochastic Navier-Stokes equations with anisotropic viscosity, we need the following spaces:
$$H:=\{u\in L^2(\mathbb{T}^2; \mathbb{R}^2); \text{div}\text{ } u=0\},$$
$$V:=\{u\in H^1(\mathbb{T}^2; \mathbb{R}^2); \text{div}\text{ } u=0\},$$
$$\tilde{H}^{s,s'}:=\{u\in H^{s,s'}(\mathbb{T}^2;\mathbb{R}^2); \text{div}\text{ } u=0\}.$$
Moreover, we use $\langle\cdot, \cdot\rangle $ to denote the scalar product (which is also the inner product of $L^2$ and $H$)
$$\langle u, v\rangle =\sum_{j=1}^2\int_{\mathbb{T}^2}u^j(x)v^j(x)dx$$
and $\langle\cdot, \cdot\rangle_X$ to denote the inner product of Hilbert space $X$ where $X=l^2$, $V$ or $\tilde{H}^{s,s'}$.

Due to the divergence free condition, we need the Larey projection operator $P_H: \text{ } L^2(\mathbb{T}^2)\rightarrow H$:
$$P_H: u\mapsto u-\nabla \Delta^{-1}(\text{div } u).$$
 
By applying the operator $P_H$ to (\ref{equation before projection})
we can rewrite the equation in the following form:
\begin{equation}\label{equation after projection}\aligned
&du(t)=\partial^2_1 u(t)dt-B(u(t))dt+\sigma(t, u(t))dW(t),\\
&u(0)=u_0,
\endaligned
\end{equation}
where the nonlinear operator $B(u,v)=P_H(u\cdot \nabla v)$ with the notation $B(u)=B(u,u)$. Here we use the same symbol  $\sigma$ after projection for simplicity.

For $u,v,w\in V$, define
$$b(u,v,w):=\langle B(u,v), w\rangle.$$
We have $b(u,v,w)=-b(u,w,v)$ and $b(u,v,v)=0$.

 We put some estimates of $b$ in the Appendix.  


 \vskip.10in
{\textbf{Large deviation principle}}

We recall the definition of the large deviation principle. For a general introduction to the theory we refer to \cite{DZ92}, \cite{DZ98}.

\begin{definition}[Large deviation principle]
Given a family of probability measures $\{\mu_\varepsilon\}_{\varepsilon>0}$ on a  metric space $(E,\rho)$ and a lower semicontinuous function $I:E\rightarrow [0,\infty]$ not identically equal to $+\infty$. The family $\{\mu_\varepsilon\}$ is said to satisfy the large deviation principle(LDP) with respect to the rate function $I$ if\\
(U) for all closed sets $F\subset E$ we have 
$$\limsup_{\varepsilon\rightarrow 0}\varepsilon\log\mu_\varepsilon(F)\leqslant-\inf_{x\in F}I(x),$$
(L) for all open sets $G\subset E$ we have 
$$\liminf_{\varepsilon\rightarrow 0}\varepsilon\log\mu_\varepsilon(G)\geqslant-\inf_{x\in G}I(x).$$

A family of random variable is said to satisfy large deviation principle if the law of these random variables satisfy large deviation princple.

Moreover, $I$ is a good rate function if its level sets $I_r:=\{x\in E:I(x)\leqslant r\}$ are compact for arbitrary $r\in (0,+\infty)$.
\end{definition}
 \vskip.10in

\begin{definition}[Laplace principle]
A sequence of random variables $\{X^\varepsilon\}$ is said to satisfy the Laplace principle with rate function $I$ if for each bounded continuous real-valued function $h$ defined on $E$
$$\lim_{\varepsilon\rightarrow 0}\varepsilon\log E\left[e^{-\frac{1}{\varepsilon}h(X^\varepsilon)}\right]=-\inf_{x\in E}\{h(x)+I(x)\}.$$
\end{definition}

Given a probabilty space $(\Omega,\mathcal{F},P)$, the random variables $\{Z_\varepsilon\}$ and $\{\overline{Z}_\varepsilon\}$ which take values in $(E,\rho)$ are called exponentially equivalent if for each $\delta>0$, 
$$\lim_{\varepsilon\rightarrow 0}\varepsilon\log P(\rho(Z_\varepsilon,\overline{Z}_\varepsilon)>\delta)=-\infty.$$
\vskip.10in
\begin{lemma}[{\cite[Theorem 4.2.13]{DZ98}}]\label{EXEQ}
If an LDP with a rate function $I(\cdot)$ holds for the random variables $\{Z_\varepsilon\}$, which are exponentially equivalent to $\{\overline{Z}_\varepsilon\}$, then the same LDP holds for $\{\overline{Z}_\varepsilon\}$.
\end{lemma}

 \vskip.10in
{\textbf{Existence and uniqueness of  solutions}}

We introduce the precise assumptions on the diffusion coefficient $\sigma$. Given a complete probability space $(\Omega, \mathcal{F}, P)$ with filtration $\{\mathcal{F}_t\}_{t\geqslant 0}$.  Let $L_2(l^2,U)$ denotes the Hilbert-Schmidt norms from $l^2$ to $U$ for a Hilbert space $U$.  We recall the following conditions for $\sigma$ from \cite{LZZ18}:

(i) \textbf{Growth condition}

There exists nonnegative constants $K'_i$, $K_i$, $\tilde{K}_i$ ($i=0,1,2$) such that for every $t\in[0,T]$:
 \vskip.10in
(A0) $\|\sigma(t,u)\|^2_{L_2(l^2, H^{-1})}\leqslant K_0'+K_1'\|u\|^2_{H}$;

(A1) $\|\sigma(t,u)\|^2_{L_2(l^2, H)}\leqslant K_0+K_1\|u\|^2_{H}+K_2\|\partial_1 u\|^2_H$;

(A2) $\|\sigma(t,u)\|^2_{L_2(l^2, H^{0,1})}\leqslant \tilde{K}_0+\tilde{K}_1\|u\|^2_{H^{0,1}}+\tilde{K}_2(\|\partial_1 u\|_H^2+\|\partial_1\partial_2u\|^2_H)$;

 \vskip.10in
(ii)\textbf{Lipschitz condition}

There exists nonnegative constants $L_1, L_2$ such that:
 \vskip.10in
(A3) $\|\sigma(t,u)-\sigma(t,v)\|^2_{L_2(l^2, H)}\leqslant L_1\|u-v\|^2_{H}+L_2\|\partial_1(u-v)\|^2_H$.

 \vskip.10in

The following theorem from \cite{LZZ18} gives the well-posedness of equation (\ref{equation after projection}):

\begin{lemma}[{\cite[Theorem 4.1, Theorem 4.2]{LZZ18}}]\label{ex. and uniq. of solution}
Under the assumptions (A0)-(A3) with $K_2<\frac{2}{11}, \tilde{K}_2<\frac{2}{5}, L_2<\frac{2}{5}$, equation (\ref{equation after projection}) has a unique strong solution $u\in L^\infty([0,T], \tilde{H}^{0,1})\cap L^2([0, T], \tilde{H}^{1,1})\cap C([0,T], H^{-1})$ for $u_0\in\tilde{H}^{0,1}$. 
\end{lemma}

 \vskip.10in
{\textbf{A martingale lemma}}

The following remarkable result is from \cite{BY82} and \cite{D76}:
\begin{lemma}\label{martingale lemma}
There exists a universal constant $c$ such that, for any $p\geqslant 2$ and for all continuous martingale $(M_t)$ with $M_0=0$ and stopping times $\tau$,
$$\|M^*_\tau\|_p\leqslant cp^\frac{1}{2}\|\langle M\rangle^\frac{1}{2}_\tau\|_p,$$
where $M^*_t=\sup_{0\leqslant s\leqslant t}|M_s|$ and $\|\cdot\|_p$ stands for the $L^p$ norm with respect to the probability space.
\end{lemma}

\section{Large deviation principle}

In this section, we consider the large deviation principle for the stochastic Navier-Stokes equations with anisotropic viscosity. We will use the weak convergence approach introduced by Budhiraja and Dupuis in \cite{BD00}. First we recall it.  The starting point is the equivalence between the large deviation principle and the Laplace principle. This result was first formulated in \cite{P93} and it is essentially a consequence of Varadhan's lemma \cite{V66} and Bryc's converse theorem \cite{B90}.

\begin{remark}
By \cite{DZ98} we have the the equivalence between the large deviation principle and the Laplace principle in completely regular topological spaces. In \cite{BD00} the authors give the weak convergence approach on a Polish space.  Since the  proof  does not depend on the separability and the completeness, the result also holds in metric spaces.
\end{remark}

Let $\{W(t)\}_{t\geqslant 0}$ be a cylindrical Wiener process on $l^2$ w.r.t. a complete filtered probability space $(\Omega, \mathcal{F}, \mathcal{F}_t, P)$ (i.e. the path of $W$ take values in $C([0,T]; U)$, where $U$ is another Hilbert space such that the embedding $l^2\subset U$ is Hilbert-Schmidt). For $\varepsilon>0$, suppose $g^\varepsilon$: $C([0,T], U)\rightarrow E$ is a measurable map and $u^\varepsilon:=g^\varepsilon(W(\cdot))$. Let

$$\mathcal{A}:=\left\{v: v \text{ is }l^2\text{-valued }\mathcal{F}_t\text{-predictable process and }\int^T_0\|v(s)(\omega)\|^2_{l^2}ds<\infty\text{ a.s.}\right\},$$
$$S_N:=\left\{\phi\in L^2([0,T],l^2):\text{ }\int^T_0\|\phi(s)\|^2_{l^2}ds\leqslant N\right\},$$
$$\mathcal{A}_N:=\left\{v\in\mathcal{A}:\text{ } v(\omega)\in S_N\text{ P-a.s.}\right\}.$$
Here we will always refer to the weak topology on $S_N$ in the following if we do not state it explicitly.

Now we formulate the following sufficient conditions for the Laplace principle of $u^\varepsilon$ as $\varepsilon\rightarrow 0$.

\begin{hyp}\label{Hyp}
There exists a measurable map $g^0: C([0,T], U)\rightarrow E$ such that the following two conditions hold:\\
1. Let $\{v^\varepsilon:\varepsilon>0\}\subset  \mathcal{A}_N$ for some $N<\infty$. If $v^\varepsilon$ converge to $v$ in distribution as $S_N$-valued random elements, then 
$$g^\varepsilon\left(W(\cdot)+\frac{1}{\sqrt{\varepsilon}}\int^\cdot_0v^\varepsilon(s)ds\right)\rightarrow g^0\left(\int^\cdot_0v(s)ds\right)$$
in distribution as $\varepsilon\rightarrow 0$.\\
2. For each $N<\infty$, the set
$$K_N=\left\{g^0\left(\int^\cdot_0\phi(s)ds\right): \phi\in S_N\right\}$$
is a compact subset of $E$. 
\end{hyp}

\begin{lemma}[{\cite[Theorem 4.4]{BD00}}]\label{weak convergence method}
If $u^\varepsilon=g^\varepsilon(W)$ satisfies the Hypothesis \ref{Hyp}, then the family $\{u^\varepsilon\}$ satisfies the Laplace principle (hence large deviation principle) on $E$ with the good rate function $I$ given by
\begin{equation}
I(f)=\inf_{\{\phi\in L^2([0,T], l^2):f=g^0(\int^\cdot_0\phi(s)ds)\}}\left\{\frac{1}{2}\int^T_0\|\phi(s)\|^2_{l^2}ds\right\}.
\end{equation}
\end{lemma}

Consider the following equation:
\begin{equation}\label{LDP eq.}\aligned
&du^\varepsilon(t)=\partial^2_1 u^\varepsilon(t)dt-B(u^\varepsilon(t))dt+\sqrt{\varepsilon}\sigma(t, u^\varepsilon(t))dW(t),\\
&u^\varepsilon(0)=u_0.
\endaligned
\end{equation}

By Lemma \ref{ex. and uniq. of solution}, under the assumptions (A0)-(A3), (\ref{LDP eq.}) has a unique strong solution $u^\varepsilon\in L^\infty([0,T], \tilde{H}^{0,1})\bigcap L^2([0,T], \tilde{H}^{1,1})\bigcap C([0,T], H^{-1})$ for $u_0\in \tilde{H}^{0,1}$. It follows from Yamada-Watanabe theorem (See \cite[Appendix E]{LR15}) that there exists a Borel-measurable function
 $$g^\varepsilon: C([0,T],U)\rightarrow L^\infty([0,T], H)\bigcap L^2([0,T], \tilde{H}^{1,0})\bigcap C([0,T],H^{-1})$$
such that $u^\varepsilon=g^\varepsilon(W)$ a.s..

Let us introduce the following skeleton equation associated to (\ref{LDP eq.}), for $\phi\in L^2([0,T], l^2)$:
\begin{equation}\label{skeleton eq.}\aligned
&dz^\phi(t)=\partial^2_1 z^\phi(t)dt-B(z^\phi(t))dt+\sigma(t, z^\phi(t))\phi(t)dt,\\
&\text{div }z^\phi=0,\\
&z^\phi(0)=u_0.
\endaligned
\end{equation}
An element $z^\phi\in L^\infty([0,T], H)\bigcap L^2([0,T], \tilde{H}^{1,0})\bigcap C([0,T],H^{-1})$ is called a (weak) solution to (\ref{skeleton eq.}) if for any $\varphi\in (C^\infty_0([0,T]\times \mathbb{T}^2))^2$ with $\text{div} \varphi=0$, and $t>0$,
$$\langle z^\phi(t),\varphi(t)\rangle=\langle u_0,\varphi(0)\rangle+\int^t_0\langle z^\phi, \partial_t\varphi\rangle-\langle \partial_1z^\phi,\partial_1 \varphi\rangle+\langle- B(z^\phi)+\sigma(s, z^\phi)\phi, \varphi\rangle ds.$$

The existence of the weak solution to (\ref{skeleton eq.}) can be obtained by the same method as in \cite{LZZ18} (see Lemma \ref{solution to skeleton eq.} in the following).

Define $g^0:C([0,T],U)\rightarrow L^\infty([0,T], H)\bigcap L^2([0,T], \tilde{H}^{1,0})\bigcap C([0,T],H^{-1})$ by
\begin{align*}
g^0(h):=
\left\{
             \begin{array}{ll}
             z^\phi, &\text{ if }h=\int^\cdot_0\phi(s)ds \text{ for some }\phi\in L^2([0,T],l^2);  \\
             0, &\text{  otherwise.} 
             \end{array}
\right.
\end{align*}
Then the rate function can be written as
\begin{equation}\label{rate function LDP}
I(z)=\inf\left\{\frac{1}{2}\int^T_0\|\phi(s)\|^2_{l^2}ds: \text{ }z=z^\phi,\text{ }\phi\in L^2([0,T],l^2)\right\},
\end{equation}
where $z\in L^\infty([0,T], H)\bigcap L^2([0,T], \tilde{H}^{1,0})\bigcap C([0,T],H^{-1})$.

The main result of this section is the following one:
\begin{Thm}\label{main result LDP}
Assume (A0)-(A3) hold with $L_2=0$ and  $u_0\in \tilde{H}^{0,1}$, then $u^\varepsilon$ satisfies a large deviation principle on $L^\infty([0,T], H)\bigcap L^2([0,T], \tilde{H}^{1,0})\bigcap C([0,T],H^{-1})$ with the good rate function $I$ given by (\ref{rate function LDP}).
\end{Thm}

The proof is divided into the following lemmas.

 \vskip.10in

\begin{lemma}\label{solution to skeleton eq.}
Assume (A0)-(A3) hold with $L_2=0$. For all $u_0\in\tilde{H}^{0,1}$ and $\phi\in L^2([0,T], l^2)$ there exists a unique solution $$z^\phi\in L^\infty([0,T], \tilde{H}^{0,1})\bigcap L^2([0,T], \tilde{H}^{1,1})\bigcap C([0,T],H^{-1})$$ to (\ref{skeleton eq.}).
\end{lemma}

\begin{proof}
First we give some a priori estimates for $z^\phi$. By taking $H$ inner product of (\ref{skeleton eq.}) with $z^\phi$ and using $\text{div }z^\phi=0$, we have
\begin{align*}
&\|z^\phi(t)\|^2_H+2\int^t_0\|\partial_1z^\phi(s)\|^2_H ds\\
=&\|u_0\|^2_H+2\int^t_0\langle z^\phi(s), \sigma(s, z^\phi(s))\phi(s)\rangle ds\\
\leqslant& \|u_0\|^2_H+2\int^t_0\|z^\phi(s)\|_H\|\sigma(s,z^\phi(s))\|_{L_2(l^2, H)}\|\phi(s)\|_{l^2}ds\\
\leqslant& \|u_0\|^2_H+2\int^t_0\left(\|z^\phi(s)\|^2_H\|\phi(s)\|^2_{l^2}+K_0+K_1\|z^\phi(s)\|^2_H+K_2\|\partial_1z^\phi(s)\|^2_H\right)ds,
\end{align*}
where we used (A1) in the last inequality.

Hence by Gronwall's inequality, we have
\begin{equation}\label{a priori z1}
\|z^\phi(t)\|^2_H+\int^t_0\|\partial_1 z^\phi(s)\|^2_{H}ds\leqslant (\|u_0\|^2_H+C)e^{C\int^t_0(\|\phi(s)\|^2_{l^2}+1)ds}.
\end{equation}

Similarly, we have
\begin{align*}
&\|z^\phi(t)\|^2_{\tilde{H}^{0,1}}+2\int^t_0(\|\partial_1z^\phi(s)\|^2_H+\|\partial_1\partial_2z^\phi(s)\|^2_H)ds\\
=&\|u_0\|^2_{\tilde{H}^{0,1}}-2\int^t_0\langle \partial_2z^\phi(s), \partial_2(z^\phi\cdot \nabla z^\phi)(s)\rangle ds+2\int^t_0\langle z^\phi(s), \sigma(s, z^\phi(s))\phi(s)\rangle_{\tilde{H}^{0,1}}ds\\
\leqslant&\|u_0\|^2_{\tilde{H}^{0,1}}+\int^t_0(\frac{1}{5}\|\partial_1\partial_2z^\phi(s)\|^2_H+C(1+\|\partial_1z^\phi(s)\|^2_H)\|\partial_2z^\phi(s)\|^2_H)ds\\
&+2\int^t_0(\|z^\phi(s)\|^2_{\tilde{H}^{0,1}}\|\phi(s)\|_{l^2}^2+\|\sigma(s, z^\phi(s))\|_{L_2(l^2, \tilde{H}^{0,1})}^2)ds,
\end{align*}
where we used Lemma \ref{estimate for b with partial_2} in the last inequality.

 Hence by (A2) we deduce that
\begin{align*}
&\|z^\phi(t)\|^2_{\tilde{H}^{0,1}}+\int^t_0\|z^\phi(s)\|^2_{\tilde{H}^{1,1}}ds\\
\leqslant& \|u_0\|^2_{\tilde{H}^{0,1}}+C+C\int^t_0(1+\|\partial_1z^\phi(s)\|^2_H+\|\phi(s)\|^2_{l^2})\|z^\phi(s)\|^2_{\tilde{H}^{0,1}}ds.
\end{align*}

Then by Gronwall's inequality and (\ref{a priori z1}) we have
\begin{equation}\label{a priori z2}
\|z^\phi(t)\|^2_{\tilde{H}^{0,1}}+\int^t_0\|z^\phi(s)\|^2_{\tilde{H}^{1,1}}ds\leqslant (\|u_0\|^2_{\tilde{H}^{0,1}}+C)e^{C(t,\phi,u_0)},
\end{equation}
where 
$$C(t,\phi,u_0)=C\left(\int^t_0(1+\|\phi(s)\|^2_{l^2})ds+(\|u_0\|^2_H+1)e^{C\int^t_0(1+\|\phi(s)\|^2_{l^2})ds}\right).$$

Now consider the following approximate equation:
\begin{equation}\label{approxi. eq.}\aligned
\left\{
             \begin{array}{lll}
             &dz_\epsilon^\phi(t)=\partial^2_1 z_\epsilon^\phi(t)dt+\epsilon^2\partial_2^2z_\epsilon^\phi(t)dt-B(z_\epsilon^\phi(t))dt+\sigma(t, z_\epsilon^\phi(t))\phi(t)dt,\\
             &\text{div} z_\epsilon^\phi=0,\\
             &z_\epsilon^\phi(0)=u_0*j_\epsilon,
             \end{array}
\right.
\endaligned
\end{equation}
where $j$ is a smooth function on $\mathbb{R}^2$ with 
$$j(x)=1, \text{ } |x|\leqslant 1;\text{ }j(x)=0,\text{ }|x|\geqslant 2,$$
and 
$$j_\epsilon(x)=\frac{1}{\epsilon^2}j(\frac{x}{\epsilon}).$$

It follows from classical theory on Navier-Stokes system that (\ref{approxi. eq.}) has a unique global smooth solution $z^\phi_\epsilon$ for any fixed $\epsilon$.  Furthermore, along the same line to (\ref{a priori z1}) and (\ref{a priori z2}) we have
\begin{equation}\label{a priori z}\aligned
&\|z_\epsilon^\phi(t)\|^2_H+\int^t_0\|\partial_1 z_\epsilon^\phi(s)\|^2_{H}ds+\epsilon^2\int^t_0\|\partial_2z_\epsilon^\phi(s)\|^2_Hds\leqslant (\|u_0\|^2_H+C)e^{C\int^t_0(\|\phi(s)\|^2_{l^2}+1)ds},\\
&\|\partial_2 z_\epsilon^\phi(t)\|^2_{H}+\int^t_0\|\partial_1\partial_2 z_\varepsilon^\phi(s)\|^2_{H}ds+\epsilon^2\int^t_0\|\partial_2^2 z_\epsilon^\phi(s)\|^2_Hds\leqslant (\|u_0\|^2_{\tilde{H}^{0,1}}+C)e^{C(t,\phi,u_0)},
\endaligned
\end{equation}

The following follows a  similar argument as in the proof of \cite[Theorem 3.1]{LZZ18}. By (\ref{a priori z}), we have $\{z_\epsilon^\phi\}_{\epsilon>0}$ is uniformly bounded in  $L^\infty([0,T], \tilde{H}^{0,1})\bigcap L^2([0,T], \tilde{H}^{1,1})$, hence bounded in $L^4([0,T],H^{\frac{1}{2}})$ (by interpolation) and $L^4([0,T], L^4(\mathbb{T}^2))$ (by Sobolev embedding). Thus $B(z^\phi_\epsilon)$ is uniformly bounded in $L^2([0,T],H^{-1})$. Let $p\in(1,\frac{4}{3})$, we have
\begin{align*}
\int^T_0\|\sigma(s,z^\phi_\epsilon(s))\phi(s)\|^p_{H^{-1}}ds\leqslant &\int^T_0\|\sigma(s,z^\phi_\epsilon(s))\|^p_{L_2(l^2,H^{-1})}\|\phi(s)\|^p_{l^2}ds\\
\leqslant &C\int^T_0(1+\|\sigma(s,z^\phi_\epsilon(s))\|^4_{L_2(l^2,H^{-1})}+\|\phi(s)\|^2_{l^2})ds\\
\leqslant&C\int^T_0(1+\|z^\phi_\epsilon(s))\|^4_{H}+\|\phi(s)\|^2_{l^2})ds<\infty,
\end{align*}
where we used Young's inequality in the second line and (A0) in the third line.
It comes out that 
\begin{equation}\label{bded in Lp(H-1)}
\{\partial_tz_\epsilon^\phi\}_{\epsilon>0}\text{  is uniformly bounded in } L^p([0,T],H^{-1}).\end{equation}
 Thus by  Aubin-Lions lemma (see \cite[Lemma 3.6]{LZZ18}),  there exists a $z^\phi\in L^2([0,T], {H})$ such that 
$$z_\epsilon^\phi\rightarrow z^\phi \text{ strongly in }L^2([0,T],H)\text{ as }\epsilon\rightarrow 0 \text{ (in the sense of subsequence)}.$$
Since $\{z_\epsilon^\phi\}_{\epsilon>0}$ is uniformly bounded in  $L^\infty([0,T], \tilde{H}^{0,1})\bigcap L^2([0,T], \tilde{H}^{1,1})$, there exists a $\tilde{z}\in L^\infty([0,T], \tilde{H}^{0,1})\bigcap L^2([0,T], \tilde{H}^{1,1})$ such that
$$z_\epsilon^\phi\rightarrow \tilde{z} \text{ weakly in }L^2([0,T],\tilde{H}^{1,1})\text{ as }\epsilon\rightarrow 0 \text{ (in the sense of subsequence)}.$$
$$z_\epsilon^\phi\rightarrow \tilde{z} \text{ weakly star in }L^\infty([0,T],\tilde{H}^{0,1})\text{ as }\epsilon\rightarrow 0 \text{ (in the sense of subsequence)}.$$
By the uniqueness of weak convergence limit, we deduce that $z^\phi=\tilde{z}$.  By (\ref{bded in Lp(H-1)}) and \cite[Theorem 2.2]{FG95}, we also have  for any $\delta>0$
$$z_\epsilon^\phi\rightarrow z^\phi \text{ strongly in }C([0,T],H^{-1-\delta})\text{ as }\epsilon\rightarrow 0 \text{ (in the sense of subsequence)}.$$

Now we use the above convergence  to prove that $z^\phi$ is a solution to (\ref{skeleton eq.}). Note that for any $\varphi\in C^\infty([0,T]\times \mathbb{T}^2)$ with $\text{div} \varphi=0$, for any $t\in[0,T]$, $z_\epsilon^{\phi}$  satisfies
\begin{equation}\label{eq of approximate solution}
\langle z^\phi_\epsilon(t), \varphi(t)\rangle =\langle u_0,\varphi(0)\rangle+\int^{t}_0\langle z^\phi_\epsilon, \partial_t\varphi\rangle-\langle \partial_1z^\phi_\epsilon,\partial_1 \varphi\rangle-\epsilon^2\langle \partial_2z^\phi_\epsilon,\partial_2 \varphi\rangle+\langle- B(z^\phi_\epsilon)+\sigma(s, z^\phi_\epsilon)\phi, \varphi\rangle ds.
\end{equation}

By \cite[Chapter 3, Lemma 3.2]{Te79} we  have 
$$\int^{t}_0\langle- B(z^{\phi}_\epsilon), \varphi\rangle ds\rightarrow \int^{t}_0\langle- B(z^\phi), \varphi\rangle ds\text{ as }\epsilon\rightarrow 0.$$

For the last term in the right hand side of (\ref{eq of approximate solution}), we have
\begin{align*}
&\int^{t}_0\langle \sigma(s, z^{\phi}_\epsilon)\phi-\sigma(s,z^\phi)\phi, \varphi\rangle ds\\
\leqslant& \int^{t}_0 \|(\sigma(s, z^{\phi}_\epsilon)-\sigma(s,z^\phi))\phi\|_{H}\|\varphi\|_Hds\\
\leqslant & C\int^{t}_0 \|\sigma(s, z^{\phi}_\epsilon)-\sigma(s,z^\phi)\|_{L_2(l^2,H)}\|\phi\|_{l^2}ds\\
\leqslant & C\left(\int^{t}_0 \|z^{\phi}_\epsilon-z^\phi\|^2_{H}ds\right)^\frac{1}{2}\left(\int^t_0\|\phi(s)\|^2_{l^2}ds\right)^\frac{1}{2},
\end{align*}
where we used H\"older's inequality and (A3) with $L_2=0$ in the last inequality.  

Thus let $\epsilon\rightarrow 0$ in (\ref{eq of approximate solution}),  we have $z^\phi\in L^\infty([0,T], \tilde{H}^{0,1})\bigcap L^2([0,T], \tilde{H}^{1,1})$ and
$$\partial_t z^\phi=\partial_1^2z^\phi-B(z^\phi)+\sigma(t,z^\phi(t))\phi.$$
Since the right hand side belongs to $L^p([0,T],H^{-1})$, we deduce that $$z^\phi\in L^\infty([0,T], \tilde{H}^{0,1})\bigcap L^2([0,T], \tilde{H}^{1,1})\bigcap C([0,T],H^{-1}).$$

For uniqueness, let $z^\phi_1, z^\phi_2\in L^\infty([0,T], \tilde{H}^{0,1})\bigcap L^2([0,T], \tilde{H}^{1,1})\bigcap C([0,T],H^{-1})$ be two solutions to (\ref{skeleton eq.})  and $w^\phi=z^\phi_1-z^\phi_2$.   Then we have
\begin{align*}
&\|w^\phi(t)\|^2_H+2\int^t_0\|\partial_1 w^\phi(s)\|^2_Hds\\
=&\|w^\phi(0)\|^2_H-2\int^t_0\langle w^\phi(s), B(z^\phi_1)(s)-B(z^\phi_2)(s)\rangle ds\\
&+2\int^t_0\langle w^\phi(s), \sigma(s, z^\phi_1(s))\phi(s)-\sigma(s, z^\phi_2(s))\phi(s)\rangle ds\\
\leqslant& \|w^\phi(0)\|^2_H-2\int^t_0b(w^\phi(s), z^\phi_2(s), w^\phi(s))ds\\
&+2\int^t_0\|w^\phi(s)\|_H\| \sigma(s, z^\phi_1(s))-\sigma(s, z^\phi_2(s))\|_{L_2(l^2,H)}\|\phi(s)\|_{l^2} ds\\
\leqslant& \|w^\phi(0)\|^2_H+\int^t_0\frac{1}{5}\|\partial_1 w^\phi(s)\|^2_Hds+C\int^t_0(1+\|z^\phi_2(s)\|^2_{\tilde{H}^{1,1}})\|w^\phi(s)\|^2_Hds\\
&+\int^t_0(\|w^\phi(s)\|^2_H \|\phi(s)\|^2_{l^2}+L_1\|w^\phi(s)\|^2_H)ds,
\end{align*}
where we used Lemma \ref{anisotropic estimate for b} in the sixth line and (A3) with $L_2=0$ in the last line.

Then by Gronwall's inequality we have 
$$\|w^\phi(t)\|^2_H\leqslant \|w^\phi(0)\|^2_He^{C\int^t_0(1+\|z^\phi_2(s)\|^2_{\tilde{H}^{1,1}}+\|\phi(s)\|^2_{l^2})ds},$$
which along with the fact that $z_2^\phi\in L^2([0,T], \tilde{H}^{1,1})$ and $\phi\in L^2([0,T],l^2)$ implies that $w^\phi(t)=0$. That is: $z_1^\phi=z_2^\phi$.
\end{proof}

The following Lemma shows that $I$ is a good rate function. The proof follows essentially the same argument as in \cite[Proposition 4.5]{WZZ}.
\begin{lemma}\label{good rate function}
Assume (A0)-(A3) hold with $L_2=0$. For all $N<\infty$, the set 
$$K_N=\left\{g^0\left(\int^\cdot_0 \phi(s)ds\right): \phi\in S_N\right\}$$
is a compact subset in $L^\infty([0,T], H)\bigcap L^2([0,T], \tilde{H}^{1,0})\bigcap C([0,T],H^{-1})$.
\end{lemma}
\begin{proof}
By definition, we have
$$K_N=\left\{z^\phi: \phi\in L^2([0,T], l^2), \text{ }\int^T_0\|\phi(s)\|^2_{l^2}ds\leqslant N\right\}.$$

Let $\{z^{\phi_n}\}$ be a sequence in $K_N$ where $\{\phi_n\}\subset S_N$. Note that (\ref{a priori z2}) implies that $z^{\phi_n}$ is uniformly bounded in $L^\infty([0,T], H^{1,0})\cap L^2([0,T], H^{1,1})$. Thus  by weak compactness of $S_N$,  a similar argument as in the proof of Lemma \ref{solution to skeleton eq.} shows that there exists $\phi\in\mathcal{S}_N$ and $z'\in L^2([0,T],H)$ such that  the following convergence hold as $n\rightarrow \infty$ (in the sense of subsequence):  

$\phi_n\rightarrow \phi$ in $\mathcal{S}_N$ weakly,

$z^{\phi_n}\rightarrow z'$ in $L^2([0,T], H^{1,0})$ weakly,

$z^{\phi_n}\rightarrow z'$ in $L^\infty([0,T], H)$ weak-star,

$z^{\phi_n}\rightarrow z'$ in $L^2([0,T], H)$ strongly.

$z^{\phi_n}\rightarrow z'$ in $C([0,T], H^{-1-\delta})$ strongly for any $\delta>0$.

Then for any $\varphi\in C^\infty([0,T]\times \mathbb{T}^2)$ with $\text{div} \varphi=0$ and for any $t\in[0,T]$, $z^{\phi_n}$  satisfies
\begin{equation}\label{eq in good rate function}
\langle z^{\phi_n}(t), \varphi(t)\rangle=\langle u_0,\varphi(0)\rangle+\int^{t}_0\langle z^{\phi_n}, \partial_t\varphi\rangle-\langle \partial_1z^{\phi_n},\partial_1 \varphi\rangle+\langle- B(z^{\phi_n})+\sigma(s, z^{\phi_n})\phi_n, \varphi\rangle ds.
\end{equation}

Let $n\rightarrow \infty$, we have
\begin{align*}
&\int^{t}_0\langle \sigma(s, z^{\phi_n})\phi_n-\sigma(s,z')\phi, \varphi\rangle ds\\
=&\int^{t}_0\langle [\sigma(s, z^{\phi_n})-\sigma(s,z')]\phi_n+\sigma(s,z')(\phi_n-\phi), \varphi\rangle ds\\
\leqslant& \int^{t}_0 \|(\sigma(s, z^{\phi_n})-\sigma(s,z'))\phi_n\|_{H}\|\varphi\|_Hds+\int^{t}_0\langle\sigma(s,z')(\phi_n-\phi), \varphi\rangle ds\\
\leqslant & C\int^{t}_0 \|\sigma(s, z^{\phi_n})-\sigma(s,z')\|_{L_2(l^2,H)}\|\phi_n\|_{l^2}ds+\int^{t}_0\langle\sigma(s,z')(\phi_n-\phi), \varphi\rangle ds\\
\leqslant & C\left(\int^{t}_0 \|z^{\phi_n}-z'\|^2_{H}ds\right)^\frac{1}{2}\left(\int^t_0\|\phi_n(s)\|^2_{l^2}ds\right)^\frac{1}{2}+\int^{t}_0\langle\sigma(s,z')(\phi_n-\phi), \varphi\rangle ds\\
\rightarrow&\text{ }0,
\end{align*}
where we used H\"older's inequality and (A3) with $L_2=0$ in the last inequality.  By \cite[Chapter 3, Lemma 3.2]{Te79} we also have 
$$\int^{t}_0\langle- B(z^{\phi_n}), \varphi\rangle ds\rightarrow \int^{t}_0\langle- B(z'), \varphi\rangle ds.$$

Then we deduce that
$$\langle z'(t), \varphi(t)\rangle=\langle u_0,\varphi(0)\rangle+\int^{t}_0\langle z', \partial_t\varphi\rangle-\langle \partial_1z',\partial_1 \varphi\rangle+\langle- B(z')+\sigma(s, z')\phi, \varphi\rangle ds,$$
which implies that  $z'$ is a solution to (\ref{skeleton eq.}). By the uniqueness of solution, we deduce that $z'=z^\phi$.

Our goal is to prove $z^{\phi_n}\rightarrow z^\phi$  in $L^\infty([0,T], H)\bigcap L^2([0,T], \tilde{H}^{1,0})\bigcap C([0,T],H^{-1})$. 

Let $w^n=z^{\phi_n}-z^\phi$, by a direct calculation, we have
\begin{align*}
&\|w^n(t)\|^2_H+2\int^t_0\|\partial_1w^n(s)\|^2_Hds\\
=&-2\int^t_0\langle w^n(s), B(z^{\phi_n})(s)-B(z^\phi)(s)\rangle ds\\
&+2\int^t_0\langle w^n(s), \sigma(s, z^{\phi_n}(s))\phi_n(s)-\sigma(s, z^\phi(s))\phi(s)\rangle ds\\
=&-2\int^t_0 b(w^n, z^\phi, w^n)(s)ds+2\int^t_0\langle w^n(s), (\sigma(s,z^{\phi_n}(s))-\sigma(s,z^\phi(s)))\phi_n(s) \rangle ds\\
&+2\int^t_0\langle w^n(s), \sigma(s, z^\phi(s))(\phi_n(s)-\phi(s))\rangle ds\\
\leqslant & \int^t_0 \frac{1}{5}\|\partial_1w^n(s)\|^2_Hds+C\int^t_0(1+\|z^\phi(s)\|^2_{\tilde{H}^{1,1}})\|w^n(s)\|^2_Hds\\
&+C\int^t_0\|w^n(s)\|^2_H\|\phi_n(s)\|_{l^2}ds\\
&+\int^t_0\|w^n(s)\|_H\|\phi_n(s)-\phi(s)\|_{l^2}(K_0+K_1\|z^\phi(s)\|^2_H+K_2\|\partial_1 z^\phi(s)\|^2_H )^\frac{1}{2} ds,
\end{align*}
where we used Lemma \ref{anisotropic estimate for b} in the sixth line, (A3) with $L_2=0$ in the seventh line and (A1) in the last line. Then we have
\begin{align*}
&\sup_{t\in[0,T]}\|w^n(t)\|^2_H+\int^T_0\|\partial_1w^n(s)\|^2_Hds\\
\leqslant &C\int^T_0(1+\|z^\phi(s)\|^2_{\tilde{H}^{1,1}})\|w^n(s)\|^2_Hds\\
+& C(\sup_{t\in[0,T]}\|z^{\phi_n}(t)\|_H+\sup_{t\in[0,T]}\|z^\phi(t)\|_H)\left(\int^T_0\|\phi_n(s)\|^2_{l^2}ds\right)^\frac{1}{2}\left(\int^T_0\|w^n(s)\|^2_Hds\right)^\frac{1}{2}\\
+& C\left(\int^T_0\|\phi_n(s)-\phi(s)\|^2_{l^2}ds\right)^\frac{1}{2}\left(\int^T_0 (1+\|z^\phi(s)\|^2_H+\|\partial_1z^\phi(s)\|^2_H)\|w^n(s)\|^2_Hds\right)^\frac{1}{2}\\
\leqslant &C\int^T_0(1+\|z^\phi(s)\|^2_{\tilde{H}^{1,1}})\|w^n(s)\|^2_Hds+C(N)\left(\int^T_0\|w^n(s)\|^2_Hds\right)^\frac{1}{2}\\
+& CN^\frac{1}{2}\left(\int^T_0 (1+\|z^\phi(s)\|^2_H+\|\partial_1z^\phi(s)\|^2_H)\|w^n(s)\|^2_Hds\right)^\frac{1}{2},
\end{align*}
where we used (\ref{a priori z1}) and the fact that $\phi_n$, $\phi$ are in $\mathcal{S}_N$.

For any $\epsilon>0$, let 
$$A_\epsilon:=\{s\in[0,T]; \|z^{\phi_n}(s)-z^\phi(s)\|_H>\epsilon\}.$$
Since $z^{\phi_n}\rightarrow z^\phi$ in $L^2([0,T], H)$ strongly, we have
$$\int^T_0\|w^n(s)\|^2_Hds\rightarrow 0,\text{ as }n\rightarrow\infty$$
and $\lim_{n\rightarrow\infty}Leb(A_\epsilon)=0$, where $Leb(B)$ means the Lebesgue measure of $B\in\mathcal{B}(\mathbb{R})$. Thus we have
\begin{align*}
&\int^T_0(1+\|z^\phi(s)\|^2_{\tilde{H}^{1,1}})\|w^n(s)\|^2_Hds\\
\leqslant&\left(\int_{ A_\epsilon}+\int_{[0,T]\setminus A_\epsilon}\right)(1+\|z^\phi(s)\|^2_{\tilde{H}^{1,1}})\|w^n(s)\|^2_Hds\\
\leqslant& C\epsilon+ 2\int_{A_\epsilon}(1+\|z^\phi(s)\|^2_{\tilde{H}^{1,1}})(\|z^{\phi_n}(s)\|^2_H+\|z^\phi(s)\|^2_H)ds\\
\leqslant&C\epsilon+ C\int_{A_\epsilon}(1+\|z^\phi(s)\|^2_{\tilde{H}^{1,1}})ds\\
\rightarrow & \text{ }C\epsilon\text{ as } n\rightarrow \infty,
\end{align*}
where we used (\ref{a priori z1}) in the forth line and (\ref{a priori z2}) in the last line. A  similar argument also implies that 
\begin{align*}
\int^T_0 (1+\|z^\phi(s)\|^2_H+\|\partial_1z^\phi(s)\|^2_H)\|w^n(s)\|^2_Hds\leqslant C\epsilon.
\end{align*}
Hence we have
\begin{align*}
\sup_{t\in[0,T]}\|w^n(t)\|^2_H+\int^T_0\|\partial_1w^n(s)\|^2_Hds\leqslant C\epsilon+C\sqrt{\epsilon} \text{ as }n\rightarrow \infty.
\end{align*}

Since $\epsilon$ is arbitrary, we obtain that 
$$z^{\phi^n}\rightarrow z^\phi\text{ strongly in }L^\infty([0,T], H)\bigcap L^2([0,T], \tilde{H}^{1,0})\bigcap C([0,T],H^{-1}).$$

\end{proof}

For next step, consider the following equation:
\begin{equation}\label{eq. for weak convergence}\aligned
dZ^\varepsilon_v(t)&=\partial^2_1 Z^\varepsilon_v(t)dt-B(Z_v^\varepsilon(t))dt+\sigma(t, Z^\varepsilon_v(t))v^\varepsilon(t)dt+\sqrt{\varepsilon}\sigma(t, Z_v^\varepsilon(t))dW(t),\\
\text{div}Z^\varepsilon_v&=0,\\
 Z^\varepsilon_v(0)&=u_0,
\endaligned
\end{equation}
where $v^\varepsilon\in\mathcal{A}_N$ for some $N<\infty$.
Here $Z^\varepsilon_v$ should have been denoted $Z^\varepsilon_{v^\varepsilon}$ and the slight abuse of notation is for simplicity. 

\begin{lemma}\label{Girsanov thm to prove existence}
Assume (A0)-(A3) hold with $L_2=0$ and $v^\varepsilon\in\mathcal{A}_N$ for some $N<\infty$. Then $Z_v^\varepsilon=g^\varepsilon\left(W(\cdot)+\frac{1}{\sqrt{\varepsilon}}\int^\cdot_0v^\varepsilon(s)ds\right)$ is the unique strong solution to (\ref{eq. for weak convergence}).
\end{lemma}
\begin{proof}
Since $v^\varepsilon\in \mathcal{A}_N$,  by the Girsanov theorem (see \cite[Appendix I]{LR15}), $\tilde{W}(\cdot):=W(\cdot)+\frac{1}{\sqrt{\varepsilon}}\int^\cdot_0v^\varepsilon(s)ds$ is an $l^2$-cylindrical Wiener-process under the probability measure
$$d\tilde{P}:=\exp\left\{-\frac{1}{\sqrt{\varepsilon}}\int^T_0v^\varepsilon(s)dW(s)-\frac{1}{2\varepsilon}\int^T_0\|v^\varepsilon(s)\|^2_{l^2}ds\right\}dP.$$
Then $(Z_v^\varepsilon, \tilde{W})$ is the solution to (\ref{LDP eq.}) on the stochastic basis $(\Omega, \mathcal{F}, \tilde{P})$. By (A0) we have
\begin{align*}
\int^T_0\|\sigma(s,Z^\varepsilon_v(s))\|_{H^{-1}}ds<\infty.
\end{align*}
Then $(Z^\varepsilon_v, W)$ satisfies the condition of the definition of weak solution (see \cite[Definition 4.1]{LZZ18}) and hence is a weak solution to (\ref{eq. for weak convergence}) on the stochastic basis $(\Omega, \mathcal{F}, {P})$ and $Z^\varepsilon_v=g^\varepsilon\left(W(\cdot)+\frac{1}{\sqrt{\varepsilon}}\int^\cdot_0v^\varepsilon(s)ds\right)$.

If $\tilde{Z^\varepsilon_v}$ and $Z^\varepsilon_v$ are two weak solutions to  (\ref{eq. for weak convergence}) on the same stochastic basis $(\Omega, \mathcal{F}, {P})$.  Let $W^\varepsilon=Z^\varepsilon_v-\tilde{Z^\varepsilon_v}$ and $q(t)=k\int^t_0(\|Z^\varepsilon_v(s)\|^2_{\tilde{H}^{1,1}}+\|v^\varepsilon(s)\|^2_{l^2})ds$ for some constant $k$. Applying It\^o's formula to $e^{-q(t)}\|W^\varepsilon(t)\|^2_H$, we have 
\begin{align*}
&e^{-q(t)}\|W^\varepsilon(t)\|^2_H+2\int^t_0e^{-q(s)}\|\partial_1 W^\varepsilon(s)\|^2_Hds\\
=&-k\int^t_0e^{-q(s)}\|W^\varepsilon(s)\|^2_H(\|Z^\varepsilon_v(s)\|^2_{\tilde{H}^{1,1}}+\|v^\varepsilon(s)\|^2_{l^2})ds-2\int^t_0e^{-q(s)}b(W^\varepsilon, Z^\varepsilon_v, W^\varepsilon)ds\\
&+2\int^t_0e^{-q(s)}\langle \sigma(s,Z^\varepsilon_v)v^\varepsilon-\sigma(s, \tilde{Z}^\varepsilon_v)v^\varepsilon, W^\varepsilon(s)\rangle ds\\
&+2\sqrt{\varepsilon}\int^t_0e^{-q(s)}\langle W^\varepsilon(s), (\sigma(s,Z^\varepsilon_v)-\sigma(s, \tilde{Z}^\varepsilon_v))dW(s)\rangle\\
&+\varepsilon\int^t_0e^{-q(s)}\|\sigma(s,Z^\varepsilon_v)-\sigma(s, \tilde{Z}^\varepsilon_v)\|^2_{L_2(l^2,H)}ds.
\end{align*}
By Lemma \ref{anisotropic estimate for b}, there exists constants $\tilde{\alpha}\in(0,1)$ and $\tilde{C}$ such that 
$$|b(W^\varepsilon, Z^\varepsilon_v, W^\varepsilon)|\leqslant \tilde{\alpha}\|\partial_1 W^\varepsilon\|^2_H+\tilde{C}(1+\|Z^\varepsilon_v\|^2_{\tilde{H}^{1,1}})\|W^\varepsilon\|^2_H.$$
We also have 
\begin{align*}
2|\langle \sigma(s,Z^\varepsilon_v)v^\varepsilon-\sigma(s, \tilde{Z}^\varepsilon_v)v^\varepsilon, W^\varepsilon\rangle|&\leqslant 2\|(\sigma(s,Z^\varepsilon_v)-\sigma(s, \tilde{Z}^\varepsilon_v))v^\varepsilon\|_H\|W^\varepsilon\|_H\\
&\leqslant \|\sigma(s,Z^\varepsilon_v)-\sigma(s, \tilde{Z}^\varepsilon_v)\|^2_{L_2(l^2,H)}+\|v^\varepsilon\|^2_{l^2}\|W^\varepsilon\|^2_H.
\end{align*}

Let $k>2\tilde{C}$ and we may  assume $\varepsilon<\frac{16}{25}$,  by (A3) with $L_2=0$ we have 
\begin{align*}
&e^{-q(t)}\|W^\varepsilon(t)\|^2_H+(2-2\tilde{\alpha})\int^t_0e^{-q(s)}\|\partial_1 W^\varepsilon(s)\|^2_Hds\\
\leqslant& C\int^t_0e^{-q(s)}\|W^\varepsilon(s)\|^2_Hds+2\sqrt{\varepsilon}\int^t_0e^{-q(s)}\langle W^\varepsilon(s),  (\sigma(s,Z^\varepsilon_v)-\sigma(s, \tilde{Z}^\varepsilon_v))dW(s)\rangle.
\end{align*}

By the Burkh\"older-Davis-Gundy's inequality (see \cite[Appendix D]{LR15}), we have 
\begin{align*}
&2\sqrt{\varepsilon}|E[\sup_{r\in[0,t]}\int^r_0e^{-q(s)}\langle W^\varepsilon(s), (\sigma(s,Z^\varepsilon_v)-\sigma(s, \tilde{Z}^\varepsilon_v))dW(s)\rangle]|\\
\leqslant &6\sqrt{\varepsilon}E\left(\int^t_0e^{-2q(s)}\|\sigma(s,Z^\varepsilon_v)-\sigma(s, \tilde{Z}^\varepsilon_v)\|^2_{L_2(l^2,H)}\|W^\varepsilon(s)\|^2_Hds\right)^\frac{1}{2}\\
\leqslant& \sqrt{\varepsilon}E(\sup_{s\in[0,t]}(e^{-q(s)}\|W^\varepsilon(s)\|^2_H))+9\sqrt{\varepsilon} E\int^t_0e^{-q(s)}L_1\|W^\varepsilon(s)\|^2_Hds,
\end{align*}
where we used (A3) with $L_2=0$ and assume that  $\tilde{\alpha}<1$.

Thus we have
$$E(\sup_{s\in[0,t]}(e^{-q(s)}\|W^\varepsilon(s)\|^2_H))\leqslant CE\int^t_0e^{-q(s)}\|W^\varepsilon(s)\|^2_Hds.$$
By the Gronwall's inequality we obtain $W^\varepsilon=0$ $P$-a.s., i.e. $\tilde{Z^\varepsilon_v}=Z^\varepsilon_v$ $P$-a.s..

 Then by the Yamada-Watanabe theorem,  we have $Z^\varepsilon_v$ is the unique strong solution to (\ref{eq. for weak convergence}).
\end{proof}

\begin{lemma}\label{estimate eq. for weak convergence}
Assume $Z^\varepsilon_v$ is a solution to (\ref{eq. for weak convergence}) with $v^\varepsilon\in\mathcal{A}_N$ and $\varepsilon<1$ small enough. Then we have
\begin{equation}\label{eq01 in lemma estimate eq. for weak convergence}
E(\sup_{t\in[0,T]}\|Z^\varepsilon_v(t)\|^4_H)+E\int^T_0\|Z^\varepsilon_v(s)\|^2_H\|Z^\varepsilon_v(s)\|^2_{\tilde{H}^{1,0}}ds+E\int^T_0\|\partial_1 Z^\varepsilon_v(s)\|^2_{H}ds\leqslant C(N,u_0).
\end{equation}
Moreover, there exists $k>0$ such that 
\begin{equation}\label{eq02 in lemma estimate eq. for weak convergence}
E(\sup_{t\in[0,T]}e^{-kg(t)}\|Z^\varepsilon_v(t)\|^2_{\tilde{H}^{0,1}})+E\int^{T}_0e^{-kg(s)}\|Z^\varepsilon_v(s)\|^2_{\tilde{H}^{1,1}}ds\leqslant C(N,u_0),
\end{equation}
where $g(t)=\int^t_0\|Z^\varepsilon_v(s)\|^2_{H}ds$ and $C(N,u_0)$ is a constant depend on $N, u_0$ but independent of $\varepsilon$.

\end{lemma}
\begin{proof}
We prove (\ref{eq01 in lemma estimate eq. for weak convergence}) by two parts of estimates. 
For first step, applying It\^o's formula to $\|Z^\varepsilon_v(t)\|^2_H$, we have
\begin{align*}
&\|Z^\varepsilon_v(t)\|^2_H+2\int^t_0\|\partial_1 Z^\varepsilon_v(s)\|^2_Hds\\
=&\|u_0\|^2_H+2\int^t_0 \langle Z^\varepsilon_v(s), \sigma(s,Z^\varepsilon_v(s))v^\varepsilon(s)\rangle ds\\
&+2\sqrt{\varepsilon}\int^t_0\langle Z^\varepsilon_v(s), \sigma(s,Z^\varepsilon_v(s))dW(s)\rangle+\varepsilon\int^t_0\|\sigma(s,Z^\varepsilon_v(s))\|^2_{L_2(l^2, H)}ds\\
\leqslant&\|u_0\|^2_H+ \int^t_0(\|Z^\varepsilon(s)\|^2_H\|v^\varepsilon(s)\|^2_{l^2}+\|\sigma(s, Z^\varepsilon_v(s))\|^2_{L_2(l^2,H)})ds\\
&+2\sqrt{\varepsilon}\int^t_0\langle Z^\varepsilon_v(s), \sigma(s,Z^\varepsilon_v(s))dW(s)\rangle+\varepsilon\int^t_0\|\sigma(s,Z^\varepsilon_v(s))\|^2_{L_2(l^2, H)}ds\\
\leqslant &\|u_0\|^2_H+\int^t_0\|Z^\varepsilon_v(s)\|^2_H\|v^\varepsilon(s)\|^2_{l^2}ds+ (1+\varepsilon)\int^t_0(K_0+K_1\|Z^\varepsilon_v\|^2_H+K_2\|\partial_1Z^\varepsilon_v\|_H^2)ds\\
&+2\sqrt{\varepsilon}\int^t_0\langle Z^\varepsilon_v(s), \sigma(s,Z^\varepsilon_v(s))dW(s)\rangle,
\end{align*}
where we used (A1) in the last inequality.

By Gronwall's inequality and $v^\varepsilon\in \mathcal{A}_N$, 
\begin{align*}
&\|Z^\varepsilon_v(t)\|^2_H+(2-(1+\varepsilon)K_2)\int^t_0\|\partial_1 Z^\varepsilon_v(s)\|^2_Hds\\
\leqslant &(\|u_0\|^2_H+C+2\sqrt{\varepsilon}\int^t_0\langle Z^\varepsilon_v(s), \sigma(s,Z^\varepsilon_v(s))dW(s)\rangle)e^{N+2K_1T}.
\end{align*}
For the term in the right hand side, by the Burkh\"older-Davis-Gundy inequality we have

\begin{align*}
&2\sqrt{\varepsilon}e^{N+K_1T} E\left(\sup_{0\leqslant s\leqslant t}|\int^s_0\langle Z^\varepsilon_v(r), \sigma(r,Z^\varepsilon_v(r))dW(r)\rangle|\right)\\
\leqslant &6\sqrt{\varepsilon}e^{N+K_1T}E\left(\int^t_0\|Z^\varepsilon_v(r)\|^2_H\|\sigma(r,Z^\varepsilon_v(r))\|^2_{L_2(l^2,H)}ds\right)^\frac{1}{2}\\
\leqslant &\sqrt{\varepsilon} E[\sup_{0\leqslant s\leqslant t}(\|Z^\varepsilon_v(s)\|^2_H)]+9\sqrt{\varepsilon}e^{2N+2K_1T}E\int^t_0[K_0+K_1\|Z^\varepsilon_v(s)\|^2_H+K_2\|\partial_1Z^\varepsilon_v(s)\|^2_H]ds,
\end{align*}
where  $(9\sqrt{\varepsilon}e^{2N+2K_1T}+1+\varepsilon)K_2-2<0$ (this can be done when $\varepsilon<(\frac{10}{9e^{2N+2K_1T}+1})^2$) and we used (A1) in the last inequality. Thus we have

\begin{align*}
&E[\sup_{s\in[0,t]}(\|Z^\varepsilon_v(t)\|^2_H)]+E\int^t_0\|\partial_1 Z^\varepsilon_v(s)\|^2_Hds\\
\leqslant&C(\|u_0\|^2_H+1)+C\int^t_0E[\sup_{r\in[0,s]}(\|Z^\varepsilon_v(r)\|^2_H)]ds.
\end{align*}

Then by  Gronwall's inequality  we have
\begin{equation}\label{eq1 in lemma estimate eq. for weak convergence}\aligned
E(\sup_{0\leqslant t\leqslant T}\|Z^\varepsilon_v(t)\|^2_H)+E\int^T_0\|\partial_1 Z^\varepsilon_v(s)\|^2_{H}ds\leqslant C(1+\|u_0\|^2_H).
\endaligned
\end{equation}

The second step is similar to \cite[Lemma 4.2]{LZZ18}. By It\^o's formula we have
\begin{equation}\label{L4 estimate step1}\aligned
\|Z^\varepsilon_v(t)\|^4_H=&\|u_0\|^4_H-4\int^t_0\|Z^\varepsilon_v\|^2_H\|\partial_1Z^\varepsilon_v(s)\|^2_Hds\\
&+4\int^t_0\|Z^\varepsilon_v(s)\|_H^2\langle\sigma(s,Z^\varepsilon_v(s))v^\varepsilon(s),Z^\varepsilon_v(s)\rangle ds\\
&+2\varepsilon\int^t_0\|Z^\varepsilon_v(s)\|^2_H\|\sigma(s,Z^\varepsilon_v(s))\|^2_{L_2(l^2,H)}ds\\
&+4\varepsilon\int^t_0\|\sigma(s,Z^\varepsilon_v(s))^*(Z^\varepsilon_v)\|^2_{l^2}ds\\
&+4\sqrt{\varepsilon}\int^t_0\|Z^\varepsilon_v(s)\|^2_H\langle Z^\varepsilon_v(s), \sigma(s,Z^\varepsilon_v(s))dW(s)\rangle_H\\
=:&\|u_0\|^4_H-4\int^t_0\|Z^\varepsilon_v\|^2_H\|\partial_1Z^\varepsilon_v(s)\|^2_Hds+I_1+I_2+I_3+I_4.
\endaligned
\end{equation}
By (A1) we have 
\begin{align*}
I_1(t)\leqslant& 4\int^t_0\|Z^\varepsilon_v(s)\|^2_H\|\sigma(s,Z^\varepsilon_v(s))\|_{L_2(l^2,H)}\|v^\varepsilon(s)\|_{l^2}\|Z^\varepsilon_v(s)\|_Hds\\
\leqslant& 2\int^t_0\|Z^\varepsilon_v(s)\|^2_H(K_0+K_1\|Z^\varepsilon_v(s)\|^2_H+K_2\|\partial_1Z^\varepsilon_v(s)\|^2_H+\|v^\varepsilon(s)\|^2_{l^2}\|Z^\varepsilon_v(s)\|^2_H)ds,
\end{align*}
and 
\begin{align*}
I_2+I_3\leqslant &6\varepsilon\int^t_0\|\sigma(s,Z^\varepsilon_v(s))\|_{L_2(l^2,H)}^2\|Z^\varepsilon_v(s)\|_H^2ds\\
\leqslant& 6\varepsilon\int^t_0(K_0+K_1\|Z^\varepsilon_v(s)\|^2_H+K_2\|\partial_1Z^\varepsilon_v(s)\|^2_H)\|Z^\varepsilon_v(s)\|^2_Hds.
\end{align*}
Thus we have
\begin{align*}
&\|Z^\varepsilon_v(t)\|^4_H+(4-2K_2-6\varepsilon K_2)\int^t_0\|Z^\varepsilon_v(s)\|^2_H\|\partial_1Z^\varepsilon_v(s)\|^2_Hds\\
\leqslant& \|u_0\|^4_H+I_4+(2+6\varepsilon)K_0\int^t_0\|Z^\varepsilon_v(s)\|^2_Hds+\int^t_0(2K_1+6\varepsilon K_1+2\|v^\varepsilon(s)\|^2_{l^2})\|Z^\varepsilon_v(s)\|^4_H)ds.
\end{align*}
Since $v^\varepsilon\in\mathcal{A}_N$, by Gronwall's inequality we have
\begin{align*}
&\|Z^\varepsilon_v(t)\|^4_H+(4-2K_2-6\varepsilon K_2)\int^t_0\|Z^\varepsilon_v(s)\|^2_H\|\partial_1Z^\varepsilon_v(s)\|^2_Hds\\
\leqslant &\left(\|u_0\|^4_H+I_4+(2+6\varepsilon)K_0\int^t_0\|Z^\varepsilon_v(s)\|^2_Hds\right)e^{8K_1T+N}.
\end{align*}

The Burkh\"older-Davis-Gundy inequality, the Young's inequality and (A1) imply that
\begin{align*}
E(\sup_{s\in[0,t]}I_4(s))\leqslant& 12\sqrt{\varepsilon} E\left(\int^t_0\|\sigma(s,Z^\varepsilon_v(s))\|^2_{L_2(l^2,H)}\|Z^\varepsilon_v(s)\|^6_Hds\right)^\frac{1}{2}\\
\leqslant& \sqrt{\varepsilon} E(\sup_{s\in[0,t]}\|Z^\varepsilon_v(s)\|^4_H)\\
&+36\sqrt{\varepsilon}E\int^t_0(K_0+K_1\|Z^\varepsilon_v(s)\|^2_H+K_2\|\partial_1Z^\varepsilon_v(s)\|^2_H)\|Z^\varepsilon_v(s)\|^2_Hds.
\end{align*}
Let $\varepsilon$ small enough such that $2K_2+6\varepsilon K_2+36\sqrt{\varepsilon}K_2e^{8K_1T+N}<4$ and $\sqrt{\varepsilon}e^{8K_1T+N}<1$ (for instance $\varepsilon<(\frac{10}{3+18e^{8K_1T+N}})^2$). Then the above estimates and (\ref{eq01 in lemma estimate eq. for weak convergence}) imply that
\begin{align*}
&E(\sup_{s\in[0,t]}\|Z^\varepsilon_v(s)\|^4_H)+\int^t_0 \|Z^\varepsilon_v(s)\|^2_H\|Z^\varepsilon_v(s)\|^2_{\tilde{H}^{1,0}}ds\\
\leqslant &C(N,u_0)+CE\int^t_0\|Z^\varepsilon_v(s)\|^4_Hds,
\end{align*}
which by Gronwall's inequality yields that 
\begin{align*}
E(\sup_{s\in[0,t]}\|Z^\varepsilon_v(s)\|^4_H)+\int^t_0 \|Z^\varepsilon_v(s)\|^2_H\|Z^\varepsilon_v(s)\|^2_{\tilde{H}^{1,0}}ds\leqslant C(N,u_0).
\end{align*}

For (\ref{eq02 in lemma estimate eq. for weak convergence}), let $h(t)=kg(t)+\int^t_0\|v^\varepsilon(s)\|^2_{l^2}ds$ for some universal constant $k$. Applying It\^o's formula to $e^{-h(t)}\|Z^\varepsilon_v(t)\|^2_{\tilde{H}^{0,1}}$, we have
\begin{align*}
&e^{-h(t)}\|Z^\varepsilon_v(t)\|^2_{\tilde{H}^{0,1}}+2\int^t_0e^{-h(s)}(\|\partial_1 Z^\varepsilon_v(s)\|^2_H+\|\partial_1\partial_2 Z^\varepsilon_v(s)\|^2_H)ds\\
=&\|u_0\|^2_{\tilde{H}^{0,1}}-\int^t_0e^{-h(s)}(k\|\partial_1Z^\varepsilon_v(s)\|^2_H+\|v^\varepsilon(s)\|^2_{l^2})\|Z^\varepsilon_v(s)\|^2_{\tilde{H}^{0,1}}ds\\
&+2\int^t_0e^{-h(s)}\langle \partial_2 Z^\varepsilon_v(s), \partial_2(Z^\varepsilon_v\cdot \nabla Z^\varepsilon_v)(s)\rangle ds+2\int^t_0e^{-h(s)}\langle Z^\varepsilon_v(s), \sigma(s,Z^\varepsilon_v(s))v^\varepsilon(s)\rangle_{\tilde{H}^{0,1}}ds\\
&+2\sqrt{\varepsilon}\int^t_0 e^{-h(s)}\langle Z^\varepsilon_v(s), \sigma(s,Z^\varepsilon_v(s))dW(t)\rangle_{\tilde{H}^{0,1}}+\varepsilon\int^t_0e^{-h(s)}\|\sigma(s,Z^\varepsilon_v(s))\|^2_{L_2(l^2, \tilde{H}^{0,1})}ds.
\end{align*}
By Lemma \ref{estimate for b with partial_2},  there exists a constant $C_1$ such that
$$|\langle \partial_2 Z^\varepsilon_v, \partial_2(Z^\varepsilon_v\cdot \nabla Z^\varepsilon_v)\rangle|\leqslant \frac{1}{2}\|\partial_1\partial_2Z^\varepsilon_v\|^2_H+C_1(1+\|\partial_1Z^\varepsilon_v\|^2_H)\|\partial_2Z^\varepsilon_v\|^2_H.$$
By Young's inequality, 
$$2|\langle Z^\varepsilon_v(s), \sigma(s,Z^\varepsilon_v(s))v^\varepsilon(s)\rangle_{\tilde{H}^{0,1}}|\leqslant \|Z^\varepsilon_v\|^2_{\tilde{H}^{0,1}}\|v^\varepsilon\|^2_{l^2}+\|\sigma(s,Z^\varepsilon_v)\|^2_{L_2(l^2,\tilde{H}^{0,1})}.$$
Choosing $k>2C_1$, we have
\begin{align*}
&e^{-h(t)}\|Z^\varepsilon_v(t)\|^2_{\tilde{H}^{0,1}}+\int^t_0e^{-h(s)}(\|\partial_1 Z^\varepsilon_v(s)\|^2_H+\|\partial_1\partial_2 Z^\varepsilon_v(s)\|^2_H)ds\\
\leqslant&\|u_0\|^2_{\tilde{H}^{0,1}}+C\int^t_0e^{-h(s)}\|\partial_2Z^\varepsilon_v(s)\|^2_Hds+(1+\varepsilon)\int^t_0e^{-h(s)}\|\sigma(s,Z^\varepsilon_v(s))\|^2_{L_2(l^2, \tilde{H}^{0,1})}ds\\
&+2\sqrt{\varepsilon}\int^t_0 e^{-h(s)}\langle Z^\varepsilon_v(s), \sigma(s,Z^\varepsilon_v(s))dW(t)\rangle_{\tilde{H}^{0,1}}.
\end{align*}
By the Burkh\"older-Davis-Gundy inequality we have
\begin{align*}
&2\sqrt{\varepsilon} E\left(\sup_{s\in[0,t]}|\int^s_0e^{-h(r)}\langle Z^\varepsilon_v(r), \sigma(r,Z^\varepsilon_v(r))dW(r)\rangle_{\tilde{H}^{0,1}}|\right)\\
\leqslant &6\sqrt{\varepsilon} E\left(\int^{t}_0e^{-2h(s)}\|Z^\varepsilon_v(r)\|^2_{\tilde{H}^{0,1}}\|\sigma(r,Z^\varepsilon_v(r))\|^2_{L_2(l^2,\tilde{H}^{0,1})}ds\right)^\frac{1}{2}\\
\leqslant &\sqrt{\varepsilon} E[\sup_{s\in[0,t]}(e^{-h(s)}\|Z^\varepsilon_v(s)\|^2_{\tilde{H}^{0,1}})]\\
&+9\sqrt{\varepsilon} E\int^{t}_0e^{-h(s)}[\tilde{K_0}+\tilde{K_1}\|Z^\varepsilon_v(s)\|^2_{\tilde{H}^{0,1}}+\tilde{K_2}(\|\partial_1 Z^\varepsilon_v(s)\|^2_H+\|\partial_1\partial_2 Z^\varepsilon_v(s)\|^2_H)]ds,
\end{align*}
where  $(9\sqrt{\varepsilon}+1+\varepsilon)\tilde{K_2}-1<0$ (this can be done if $\varepsilon<\frac{9}{400}$) and we used (A2) in the last inequality.

Combine the above estimates, we have
\begin{align*}
&E(\sup_{s\in[0,{t}]}e^{-h(s)}\|Z^\varepsilon_v(s)\|^2_{\tilde{H}^{0,1}})+E\int^{t}_0e^{-h(s)}\|Z^\varepsilon_v(s)\|^2_{\tilde{H}^{1,1}}ds\\
\leqslant &C(\|u_0\|^2_{\tilde{H}^{0,1}}+1+E\int^{t}_0e^{-h(s)}\|Z^\varepsilon_v(s)\|^2_{\tilde{H}^{0,1}}ds)
\end{align*} 

Then Gronwall's inequality implies that 
\begin{align*}
E(\sup_{0\leqslant t\leqslant T}e^{-h(t)}\|Z^\varepsilon_v(t)\|^2_{\tilde{H}^{0,1}})+E\int^{T}_0e^{-h(s)}\|Z^\varepsilon_v(s)\|^2_{\tilde{H}^{1,1}}ds\leqslant C(1+\|u_0\|^2_{\tilde{H}^{0,1}}).
\end{align*}
Since $v^\varepsilon\in \mathcal{S}_N$, we deduce that
\begin{equation}\label{eq2 in lemma estimate eq. for weak convergence}
E(\sup_{t\in[0,T]}e^{-kg(t)}\|Z^\varepsilon_v(t)\|^2_{\tilde{H}^{0,1}})+E\int^{T}_0e^{-kg(s)}\|Z^\varepsilon_v(s)\|^2_{\tilde{H}^{1,1}}ds\leqslant C(1+\|u_0\|^2_{\tilde{H}^{0,1}})e^{N}.
\end{equation}

\end{proof}

Similar as \cite[lemma 4.3]{LZZ18}, we have the following tightness lemma:
\begin{lemma}\label{tightness lemma}
Assume $Z^\varepsilon_v$ is a solution to (\ref{eq. for weak convergence}) with $v^\varepsilon\in\mathcal{A}_N$ and $\varepsilon<1$ small enough. 
There exists $\varepsilon_0>0$, such that
$\{Z^\varepsilon_v\}_{\varepsilon\in(0,\varepsilon_0)}$ is tight in the space
$$\chi=C([0,T],H^{-1})\bigcap L^2([0,T],H)\bigcap L^2_w([0,T], H^{1,1})\bigcap L^\infty_{w^*}([0,T], H^{0,1}),$$
where $L^2_w$ denotes the weak topology and $L^\infty_{w^*}$ denotes the weak star topology.
\end{lemma}

\begin{proof}
Let $k$ be the same constant as in the proof of (\ref{eq02 in lemma estimate eq. for weak convergence}) and let
\begin{align*}
K_R:=&\Big{\{} u\in C([0,T],H^{-1}): \sup_{t\in[0,T]}\|u(t)\|^2_H+\int^T_0\|u(t)\|^2_{\tilde{H}^{1,0}}dt+\|u\|_{C^\frac{1}{16}([0,T], H^{-1})}\\
&+\sup_{t\in[0,T]}e^{-k\int^t_0\|\partial_1u(s)\|^2_Hds}\|u(t)\|^2_{\tilde{H}^{0,1}}+\int^T_0e^{-k\int^t_0\|\partial_1u(s)\|^2_Hds}\|u(t)\|^2_{\tilde{H}^{1,1}}dt\leqslant R\Big{\}},
\end{align*}
where $C^\frac{1}{16}([0,T], H^{-1})$ is the H\"older space with the norm:
$$\|f\|_{C^\frac{1}{16}([0,T], H^{-1})}=\sup_{0\leqslant s< t\leqslant T}\frac{\|f(t)-f(s)\|_{H^{-1}}}{|t-s|^{\frac{1}{16}}}.$$

Then from the proof of \cite[Lemma 4.3]{LZZ18}, we know that for any $R>0$, $K_R$ is relatively compact in $\chi$. 

Now we only need to show that  for any $\delta>0$, there exists $R>0$, such that $P(Z^\varepsilon_v\in  K_R)>1-\delta$ for any $\varepsilon\in(0,\varepsilon_0)$, where $\varepsilon_0$ is the constant such that Lemma \ref{estimate eq. for weak convergence} hold.

By Lemma \ref{estimate eq. for weak convergence} and Chebyshev inequality, we can choose $R_0$ large enough such that
\begin{align*}
P\left(\sup_{t\in[0,T]}\|Z^\varepsilon_v(t)\|^2_H+\int^T_0\|Z^\varepsilon_v(t)\|^2_{\tilde{H}^{1,0}}dt>\frac{R_0}{3}\right)<\frac{\delta}{4},
\end{align*}
and
\begin{align*}
P\left(\sup_{t\in[0,T]}e^{-k\int^t_0\|\partial_1u(s)\|^2_Hds}\|u(t)\|^2_{\tilde{H}^{0,1}}+\int^T_0e^{-k\int^t_0\|\partial_1u(s)\|^2_Hds}\|u(t)\|^2_{\tilde{H}^{1,1}}dt>\frac{R_0}{3}\right)<\frac{\delta}{4},
\end{align*}
where $k$ is the same constant as  in (\ref{eq02 in lemma estimate eq. for weak convergence}).

 Fix $R_0$ and let
 \begin{align*}
 \hat{K}_{R_0}=&\Big{\{}u\in C([0,T],H^{-1}): \sup_{t\in[0,T]}\|u(t)\|^2_H+\int^T_0\|u(t)\|^2_{\tilde{H}^{1,0}}dt\leqslant \frac{R_0}{3}\text{ and }\\
&\sup_{t\in[0,T]}e^{-k\int^t_0\|\partial_1u(s)\|^2_Hds}\|u(t)\|^2_{\tilde{H}^{0,1}}+\int^T_0e^{-k\int^t_0\|\partial_1u(s)\|^2_Hds}\|u(t)\|^2_{\tilde{H}^{1,1}}dt\leqslant \frac{R_0}{3}\Big{\}}.
 \end{align*}
Then $P(Z^\varepsilon_v\in C([0,T], H^{-1})\setminus \hat{K}_{R_0})<\frac{\delta}{2}$.

 Now for  $Z^\varepsilon_v\in\hat{K}_{R_0}$, we have $\partial_1^2Z^\varepsilon_v$ is uniformly bounded in $L^2([0,T],H^{-1})$. Similar as in Lemma \ref{solution to skeleton eq.}, $Z^\varepsilon_v$ is uniformly bounded in $L^4([0,T],H^\frac{1}{2})$ and $L^4([0,T], L^4(\mathbb{T}^2))$, thus $B(Z^\varepsilon_v)$ is uniformly bounded in $L^2([0,T],H^{-1})$.  By H\"older's inequality, we have
 \begin{align*}
 \sup_{s,t\in[0,T],s\neq t}\frac{\|\int^t_s\partial_1^2Z^\varepsilon_v(r)+B(Z^\varepsilon_v(r))dr\|^2_{H^{-1}}}{|t-s|}\leqslant \int^T_0\|\partial_1^2Z^\varepsilon_v(r)+B(Z^\varepsilon_v(r))\|^2_{H^{-1}}dr\leqslant C(R_0),
 \end{align*}
 where $C(R_0)$ is a constant depend on $R_0$.  For any $p\in(1,\frac{4}{3})$, by H\"older's inequality, we have
\begin{align*}
\sup_{s,t\in[0,T],s\neq t}\frac{\|\int^t_s\sigma(r,Z^\varepsilon_v(r))v^\varepsilon(r) dr\|^p_{H^{-1}}}{|t-s|^{p-1}}\leqslant& \int^T_0\|\sigma(r,Z^\varepsilon_v(r))v^\varepsilon(r)\|^p_{H^{-1}}dr\\
\leqslant& \int^T_0 \|\sigma(r,Z^\varepsilon_v(r))\|^p_{L_2(l^2,H^{-1})}\|v^\varepsilon(r)\|^p_{l^2}dr\\
\leqslant& C\int^T_0(1+\|Z^\varepsilon_v(r)\|^4_{H}+\|v^\varepsilon(r)\|^4_{l^2})dr\\
\leqslant& C(R_0),
\end{align*}
where we used Young's inequality and (A0) in the third inequality.

Moreover, for any $0\leqslant s\leqslant t\leqslant T$, by H\"older's inequality we have
\begin{align*}
E\|\int^t_s\sigma(r,Z^\varepsilon_v(r))dW(r)\|^4_{H^{-1}}\leqslant& C E\left(\int^t_s \|\sigma(r,Z^\varepsilon_v(r))\|^2_{L_2(l^2,H^{-1})}dr \right)^2\\
\leqslant& C|t-s|E\int^t_s \|\sigma(r,Z^\varepsilon_v(r))\|^4_{L_2(l^2,H^{-1})}dr\\
\leqslant& C|t-s|^2(1+E(\sup_{t\in[0,T]}\|Z^\varepsilon_v(t)\|^4_H))\\
\leqslant &C|t-s|^2,
\end{align*}
where we used (A0) in the third inequality and (\ref{eq01 in lemma estimate eq. for weak convergence}) in the last inequality. Then by Kolmogorov's continuity criterion, for any $\alpha\in(0, \frac{1}{4})$, we have
\begin{align*}
E\left(\sup_{s,t\in[0,T],s\neq t}\frac{\|\int^t_s\sigma(r,Z^\varepsilon_v(r))dW(r)\|^4_{H^{-1}}}{|t-s|^{2\alpha}}\right)\leqslant C.
\end{align*}
Choose $p=\frac{8}{7}, \alpha=\frac{1}{8}$ in the above estimates, we deduce that there exists $R>R_0$ such that 
\begin{align*}
P\left(\|Z^\varepsilon_v\|_{C^\frac{1}{16}([0,T],H^{-1})}>\frac{R}{3}, Z^\varepsilon_v\in \hat{K}_{R_0}\right)\leqslant \frac{E\left(\sup_{s,t\in[0,T],s\neq t}\frac{\|Z^\varepsilon_v(t)-Z^\varepsilon_v(s)\|_{H^{-1}}}{|t-s|^{\frac{1}{16}}}1_{\{Z^\varepsilon_v\in \hat{K}_{R_0}\}}\right)}{\frac{R}{3}}< \frac{\delta}{2}.
\end{align*}
Combining the fact that  $P(Z^\varepsilon_v\in C([0,T], H^{-1})\setminus \hat{K}_{R_0})<\frac{\delta}{2}$, we finish the proof.

\end{proof}

\begin{lemma}\label{weak convergence}
Assume (A0)-(A3) hold with $L_2=0$. Let $\{v^\varepsilon\}_{\varepsilon>0}\subset \mathcal{A}_N$ for some $N<\infty$. Assume $v^\varepsilon$ converge to $v$ in distribution as $S_N$-valued random elements, then
\begin{align*}
g^\varepsilon\left(W(\cdot)+\frac{1}{\sqrt{\varepsilon}}\int^\cdot_0v^\varepsilon(s)ds\right)\rightarrow g^0\left(\int^\cdot_0v(s)ds\right)
\end{align*}
in distribution as $\varepsilon\rightarrow 0$.
\end{lemma}

\begin{proof}
The proof follows essentially the same argument as in \cite[Proposition 4.7]{WZZ}. 

By Lemma \ref{Girsanov thm to prove existence}, we have $Z_v^\varepsilon=g^\varepsilon\left(W(\cdot)+\frac{1}{\sqrt{\varepsilon}}\int^\cdot_0v^\varepsilon(s)ds\right)$. By a similar but simple argument as in the proof of Lemmas \ref{solution to skeleton eq.} and \ref{estimate eq. for weak convergence}, there exists a unique strong solution $Y^\varepsilon\in L^{\infty}([0,T],\tilde{H}^{0,1})\bigcap L^2([0,T],\tilde{H}^{1,1})\bigcap C([0,T],H^{-1})$ satisfying
\begin{align*}
dY^\varepsilon(t)=&\partial_1^2Y^\varepsilon(t)dt+\sqrt{\varepsilon}\sigma(t,Z^\varepsilon_v(t))dW(t),\\
\text{div }Y^\varepsilon =&0,\\
Y^\varepsilon(0)=&0,
\end{align*}
and 
\begin{align*}
\lim_{\varepsilon\rightarrow 0}\left[E\sup_{t\in[0,T]}\|Y^\varepsilon(t)\|_{H}^2+E\int^T_0\|Y^\varepsilon(t)\|^2_{\tilde{H}^{1,0}}dt\right]=0,
\end{align*}
\begin{align*}
\lim_{\varepsilon\rightarrow 0}\left[E\sup_{t\in[0,T]}(e^{-kg(t)}\|Y^\varepsilon(t)\|_{\tilde{H}^{0,1}}^2)+E\int^T_0e^{-kg(t)}\|Y^\varepsilon(t)\|^2_{\tilde{H}^{1,1}}dt\right]=0,
\end{align*}
where $g(t)=\int^t_0\|Z^\varepsilon_v(s)\|^2_Hds$ and $k$ are the same as in (\ref{eq02 in lemma estimate eq. for weak convergence}).

Set $$\Xi:=\left(\chi, \mathcal{S}_N,  L^\infty([0,T],H)\bigcap L^2([0,T],\tilde{H}^{1,0}) \bigcap C([0,T].H^{-1})\right).$$ The above limit implies that $Y^\varepsilon\rightarrow 0$ a.s. in $ L^\infty([0,T],H)\bigcap L^2([0,T],\tilde{H}^{1,0}) \bigcap C([0,T].H^{-1})$ as $\varepsilon\rightarrow 0$ (in the sense of subsequence). By Lemma \ref{tightness lemma} the family $\{(Z^\varepsilon_v,v^\varepsilon)\}_{\varepsilon\in(0,\varepsilon_0)}$ is tight in $(\chi, \mathcal{S}_N)$. Let $(Z_v, v, 0)$ be any limit point of $\{(Z^\varepsilon_v, v^\varepsilon, Y^\varepsilon)\}_{\varepsilon\in(0,\varepsilon_0)}$. Our goal is to show that $Z_v$ has the same law as $g^0\left(\int^\cdot_0v(s)ds\right)$ and $Z^\varepsilon_v$ convergence in distribution to $Z_v$ in the space $ L^{\infty}([0,T],H)\bigcap L^2([0,T],\tilde{H}^{1,0})\bigcap C([0,T],H^{-1})$.

By the Skorokhod Theorem, there exists a stochastic basis $(\tilde{\Omega}, \tilde{\mathcal{F}}, \{\tilde{\mathcal{F}}_t\}_{t\in[0,T]}, \tilde{P})$ and, on this basis, $\Xi$-valued random variables $(\tilde{Z}_v, \tilde{v}, 0)$, $(\tilde{Z}^\varepsilon_v, \tilde{v}^\varepsilon, \tilde{Y}^\varepsilon)$, such that $(\tilde{Z}^\varepsilon_v, \tilde{v}^\varepsilon, \tilde{Y}^\varepsilon)$ (respectively $(\tilde{Z}_v, \tilde{v}, 0)$) has the same law as $(Z^\varepsilon_v,v^\varepsilon, Y^\varepsilon)$ (respectively $(Z_v,v,0)$), and $(\tilde{Z}^\varepsilon_v, \tilde{v}^\varepsilon, \tilde{Y}^\varepsilon)\rightarrow (\tilde{Z}_v, \tilde{v}, 0)$, $\tilde{P}$-a.s.

We have
\begin{equation}\label{weak convergence step 1}\aligned
d(\tilde{Z}^\varepsilon_v(t)-\tilde{Y}^\varepsilon(t))=&\partial_1^2(\tilde{Z}^\varepsilon_v(t)-\tilde{Y}^\varepsilon(t))dt-B(\tilde{Z}^\varepsilon_v(t))dt+\sigma(t,\tilde{Z}^\varepsilon_v(t))\tilde{v}^\varepsilon(t)dt,\\
\tilde{Z}^\varepsilon_v(0)-\tilde{Y}^\varepsilon(0)=&u_0,
\endaligned
\end{equation}
and
\begin{align*}
&P(\tilde{Z}^\varepsilon_v-\tilde{Y}^\varepsilon\in L^\infty([0,T],H)\bigcap L^2([0,T], \tilde{H}^{1,0})\bigcap C([0,T],H^{-1}))\\
=&P({Z}^\varepsilon_v-{Y}^\varepsilon\in L^\infty([0,T],H)\bigcap L^2([0,T], \tilde{H}^{1,0})\bigcap C([0,T],H^{-1}))\\
=&1.
\end{align*}

Let $\tilde{\Omega}_0$ be the subset of $\tilde{\Omega}$ such that for $\omega\in\tilde{\Omega}_0$,
$$(\tilde{Z}^\varepsilon_v, \tilde{v}^\varepsilon, \tilde{Y}^\varepsilon)(\omega)\rightarrow (\tilde{Z}_v,\tilde{v}, 0)(\omega) \text{ in }\Xi,$$ 
and 
$$e^{-k\int^{\cdot}_0\|\tilde{Z}^\varepsilon_v(\omega, s)\|^2_Hds}\tilde{Y}^\varepsilon(\omega)\rightarrow 0 \text{ in } L^{\infty}([0,T],\tilde{H}^{0,1})\bigcap L^2([0,T],\tilde{H}^{1,1})\bigcap C([0,T],H^{-1}),$$ 
then $P(\tilde{\Omega}_0)=1$. For any $\omega\in\tilde{\Omega}_0$, fix $\omega$, we have $\sup_{\varepsilon}\int^T_0\|\tilde{Z}^\varepsilon_v(\omega,s)\|_H^2ds<\infty$, then we deduce that 
\begin{equation}\label{convergence of Yvarepsilon to 0}
\lim_{\varepsilon\rightarrow 0} \left(\sup_{t\in[0,T]}\|\tilde{Y}^\varepsilon(\omega,t)\|_{\tilde{H}^{0,1}}+\int^T_0\|\tilde{Y}^\varepsilon(\omega,t)\|^2_{\tilde{H}^{1,1}}dt\right)=0.
\end{equation}

Now we  show that
\begin{equation}\label{weak convergence step 2}
\sup_{t\in[0,T]}\|\tilde{Z}^\varepsilon_v(\omega,t)-\tilde{Z}_v(\omega,t)\|^2_H+\int^T_0\|\tilde{Z}^\varepsilon_v(\omega, t)-\tilde{Z}_v(\omega,t)\|^2_{\tilde{H}^{1,0}}dt\rightarrow 0\text{ as }\varepsilon\rightarrow 0.
\end{equation}

Let  $Z^\varepsilon=\tilde{Z}^\varepsilon_v(\omega)-\tilde{Y}^\varepsilon(\omega)$, then by  (\ref{weak convergence step 1}) we have
\begin{equation}
dZ^\varepsilon(t)=\partial_1^2Z^\varepsilon(t) dt-B(Z^\varepsilon(t)+\tilde{Y}^\varepsilon(t))dt+\sigma(t,Z^\varepsilon(t)+\tilde{Y}^\varepsilon(t))\tilde{v}^\varepsilon(t)dt.
\end{equation}
Since $Z^\varepsilon(\omega)\rightarrow \tilde{Z}_v(\omega)$  in $\chi$, by a very similar argument as in Lemma \ref{good rate function} we deduce that  $\tilde{Z}_v=z^{\tilde{v}}=g^0\left(\int^\cdot_0{\tilde{v}}(s)ds\right)$. Moreover, note that $\tilde{Z}^\varepsilon_v(\omega)\rightarrow z^{\tilde{v}}(\omega)$ weak star in $L^{\infty}([0,T],\tilde{H}^{0,1})$, then the uniform boundedness principle implies that 
\begin{equation}\label{uniform bded of Z in H01}
\sup_{\varepsilon}\sup_{t\in[0,T]}\|\tilde{Z}^\varepsilon_v(\omega)\|_{\tilde{H}^{0,1}}<\infty.
\end{equation}
Let $w^\varepsilon=Z^\varepsilon-z^{\tilde{v}}$, then we have
\begin{align*}
&\|w^\varepsilon(t)\|^2_H+2\int^t_0\|\partial_1w^\varepsilon(s)\|^2_Hds\\
=&-2\int^t_0\langle w^\varepsilon(s), B(Z^\varepsilon+\tilde{Y}^\varepsilon)-B(z^{\tilde{v}})\rangle ds+2\int^t_0\langle w^\varepsilon(s), \sigma(s,Z^\varepsilon+\tilde{Y}^\varepsilon){\tilde{v}}^\varepsilon(s)-\sigma(s,z^{\tilde{v}}){\tilde{v}}(s)\rangle ds.\\
\end{align*}
By Lemmas \ref{anisotropic estimate for b} and \ref{b(u,v,w)}, we have
\begin{align*}
&\int^t_0\langle w^\varepsilon(s), B(Z^\varepsilon+\tilde{Y}^\varepsilon)-B(z^{\tilde{v}})\rangle ds\\
=&\int^t_0b(\tilde{Y}^\varepsilon,z^{\tilde{v}},w^\varepsilon)+b(\tilde{Y}^\varepsilon,\tilde{Y}^\varepsilon,w^\varepsilon)+b(w^\varepsilon,\tilde{Y}^\varepsilon+z^{\tilde{v}},w^\varepsilon)+b(z^{\tilde{v}},\tilde{Y}^\varepsilon, w^\varepsilon)ds\\
\leqslant &\int^t_0[\frac{1}{2}\|\partial_1w^\varepsilon(s)\|^2_H+\frac{1}{2}\|\tilde{Y}^\varepsilon(s)\|^2_{\tilde{H}^{1,1}}+C(1+\|z^{\tilde{v}}(s)\|_{\tilde{H}^{1,1}}^2+\|\tilde{Y}^\varepsilon(s)\|^2_{\tilde{H}^{1,1}})\|w^\varepsilon(s)\|^2_H]ds\\
&+C\int^t_0\|\tilde{Y}^\varepsilon(s)\|^2_{\tilde{H}^{1,1}}\|w^\varepsilon(s)\|_Hds\\
\leqslant& \int^t_0\frac{1}{2}\|\partial_1 w^\varepsilon(s)\|^2_Hds+C\int^t_0\|\tilde{Y}^\varepsilon(s)\|^2_{\tilde{H}^{1,1}}ds+C\int^t_0(1+\|z^{\tilde{v}}(s)\|_{\tilde{H}^{1,1}}^2)\|w^\varepsilon(s)\|^2_Hds,
\end{align*}
where we  used the fact that by \eqref{convergence of Yvarepsilon to 0} and \eqref{uniform bded of Z in H01} $w^\varepsilon$ are uniformly bounded in $L^\infty([0,T], H)$ in the last inequality.
By (A1) and (A3) with $L_2=0$ we have
\begin{align*}
&\int^t_0\langle w^\varepsilon(s), \sigma(s,Z^\varepsilon+\tilde{Y}^\varepsilon)v^\varepsilon(s)-\sigma(s,z^{\tilde{v}}){\tilde{v}}(s)\rangle ds\\
=&\int^t_0\langle w^\varepsilon(s), (\sigma(s,Z^\varepsilon+\tilde{Y}^\varepsilon)-\sigma(s,z^{\tilde{v}})) {\tilde{v}}^\varepsilon(s)\rangle ds+\int^t_0\langle w^\varepsilon(s), \sigma(s,z^{\tilde{v}})({\tilde{v}}^\varepsilon(s)- {\tilde{v}}(s))\rangle ds\\
\leqslant &C\int^t_0(\|w^\varepsilon(s)\|_H\|{\tilde{v}}^\varepsilon(s)\|_{l^2}(\|w^
\varepsilon(s)\|^2_H+\|\tilde{Y}^\varepsilon(s)\|^2_H)^\frac{1}{2} ds\\
&+\int^t_0\|w^\varepsilon(s)\|_H\|{\tilde{v}}^\varepsilon(s)-{\tilde{v}}(s)\|_{l^2}(K_0+K_1\|z^{\tilde{v}}(s)\|^2_H+K_2\|\partial_1 z^{\tilde{v}}(s)\|^2_H)^\frac{1}{2} ds\\
\leqslant &C N^\frac{1}{2}\left(\int^t_0(\|w^
\varepsilon(s)\|^2_H+\|\tilde{Y}^\varepsilon(s)\|^2_H ds\right)^\frac{1}{2}\\
&+CN^\frac{1}{2}\left(\int^t_0\|w^\varepsilon(s)\|_H^2(K_0+K_1\|z^{\tilde{v}}(s)\|^2_H+K_2\|\partial_1 z^{\tilde{v}}(s)\|^2_H) ds\right)^\frac{1}{2},
\end{align*}
where we used the fact that $w^\varepsilon$ are uniformly bounded in $L^\infty([0,T], H)$ and that ${\tilde{v}}^\varepsilon$,  ${\tilde{v}}$ are in $\mathcal{A}_N$.
Thus we have
\begin{align*}
&\|w^\varepsilon(t)\|^2_H+\int^t_0\|\partial_1w^\varepsilon(s)\|^2_Hds\\
\leqslant& C\int^t_0(1+\|z^{\tilde{v}}(s)\|_{\tilde{H}^{1,1}}^2)\|w^\varepsilon(s)\|^2_H ds+C\int^t_0\|\tilde{Y}^\varepsilon(s)\|^2_{\tilde{H}^{1,1}}ds\\
&+CN^\frac{1}{2}\left(\int^t_0(\|w^\varepsilon(s)\|^2_H+\|\tilde{Y}^\varepsilon(s)\|^2_H)ds\right)^\frac{1}{2}+CN^\frac{1}{2}\left(\int^t_0(1+\|z^{\tilde{v}}(s)\|^2_{\tilde{H}^{1,1}})\|w^\varepsilon(s)\|^2_Hds\right)^\frac{1}{2}.
\end{align*}

Since $Z^\varepsilon(\omega)\rightarrow z^{\tilde{v}}(\omega)$ strongly in $L^2([0,T],H)$ and $\tilde{Y}^\varepsilon\rightarrow 0$ in $L^2([0,T],\tilde{H}^{1,1})$, the same argument used in Lemma \ref{good rate function} implies
\begin{equation}\label{weak convergence step 3}
\sup_{t\in[0,T]}\|\tilde{Z}^\varepsilon_v(\omega,t)-z^{\tilde{v}}(\omega,t)\|^2_H+\int^T_0\|\tilde{Z}^\varepsilon_v(\omega, t)-z^{\tilde{v}}(\omega,t)\|^2_{\tilde{H}^{1,0}}dt\rightarrow 0\text{ as }\varepsilon\rightarrow 0.
\end{equation}

The proof is thus complete.

\end{proof}

\begin{proof}[{Proof of Theorem \ref{main result LDP}}]
The result holds from Lemmas \ref{weak convergence method}, \ref{good rate function} and \ref{weak convergence}.
\end{proof}

\section{Small time asymptotics}

In this section, we consider the small time behaviour. We need the following additional assumption (A3') and (A4). Note that (A3') is stronger than (A3).

(A3') $\|\sigma(t,u)-\sigma(s,v)\|^2_{L_2(l^2, H)}\leqslant L_0|t-s|^\alpha+L_1\|u-v\|^2_{H}$.

(A4) $\|\sigma(t,u)\|^2_{L_2(l^2, V)}\leqslant \overline{K}_0+\overline{K}_1\|u\|_{V}^2$.

\begin{remark} \text{ }
A typical example of $\sigma$ is similar as in \cite[Remark 4.2]{LZZ18}. For $u=(u^1,u^2)\in H^{1,1}$ and $y\in l^2$, let
$$\sigma(t,u)y=\sum^\infty_{k=1}b_kg(u)\langle y, \psi_k\rangle_{l^2},$$
where $\{\psi_k\}_{k\geqslant 0}$ is the orthonormal basis of $l^2$, $\{b_k\}_{k\geqslant 0}$ are functions from $\mathbb{T}^2$ to $\mathbb{R}$ and $g$ is a differentiable function from $\mathbb{R}^2$ to $\mathbb{R}$. Assume that  $|g(x)-g(y)|\leqslant C|x-y|$ for all $x,y\in\mathbb{R}^2$ and some constant $C$  depends on $g$.  Also suppose that $\textsl{div}(b_kg(u))=0$ and $b_k, \partial_1 b_k, \partial_2b_k\in L^\infty$, $\sum^\infty_{k=1}\|b_k\|_{L^\infty}^2\leqslant M$, $\sum^\infty_{k=1}\|\partial_1 b_k\|_{L^\infty}^2\leqslant M$ and $\sum^\infty_{k=1}\|\partial_2 b_k\|_{L^\infty}^2\leqslant M$. From the conditions of $g$, it is easy to obtain $|g(u)|\leqslant C|u|+C$, $|\partial_1g(u)|\leqslant C$ and $|\partial_2g(u)|\leqslant C$. In this case,
$\sigma$  satisfies (A0)-(A4) and (A3'):
\begin{align*}
\|\sigma(t,u)\|^2_{L_2(l^2, H)}\leqslant& \sum^\infty_{k=1}\|b_kg(u)\|^2_H\leqslant CM(\|u\|^2_H+1);\\
\|\sigma(t,u)\|^2_{L_2(l^2, H^{0,1})}\leqslant&\sum^\infty_{k=1}\|b_kg(u)\|^2_H+\sum^\infty_{k=1}\|\partial_2(b_kg(u))\|^2_H\\
\leqslant& CM(\|u\|^2_H+1)+\sum^\infty_{k=1}\|\partial_2b_kg(u)+b_k(\partial_1g(u)\partial_2u^1+\partial_2g(u)\partial_2u^2)\|^2_H\\
\leqslant& CM(1+\|u\|^2_H+\|\partial_2u\|^2_H);\\
\|\sigma(t,u)\|^2_{L_2(l^2, V)}\leqslant &C{M}(\|u\|^2_H+1)+\sum^\infty_{k=1}\|\partial_1(b_kg(u))\|^2_H+\sum^\infty_{k=1}\|\partial_2(b_kg(u))\|^2_H\\
\leqslant& C{M}(1+\|u\|^2_H+\|\partial_1u\|^2_H+\|\partial_2u\|^2_H);\\
\|\sigma(t,u)-\sigma(s,v)\|^2_{L_2(l^2, H)}\leqslant&{M}C\|u-v\|^2_H.\\
\end{align*}

\end{remark}

Let $\varepsilon>0$ and $u$ be the solution to (\ref{equation after projection}), by the scaling property of the Brownian motion, $u(\varepsilon t)$ coincides in law with the solution to the following equation:
\begin{equation}\label{eq. for small time}\aligned
&du_\varepsilon=\varepsilon\partial_1^2u_\varepsilon dt-\varepsilon B(u_\varepsilon)dt+\sqrt{\varepsilon}\sigma(\varepsilon t, u_\varepsilon)dW(t),\\
&u_\varepsilon(0)=u_0.
\endaligned
\end{equation}

Define a functional $I^{u_0}$ on $L^\infty([0,T], H)\bigcap C([0,T],H^{-1})$ by
$$I^{u_0}(g)=\inf_{h\in\Gamma_g}\{\frac{1}{2}\int^T_0\|{h}(t)\|_{l^2}^2dt\},$$
where 
$$\Gamma_g=\{h\in L^2([0,T],l^2): g(t)=u_0+\int^t_0\sigma(0,g(s)){h}(s)ds, \text{ }t\in[0,T]\}.$$

The main theorem of this section is the following one:
\begin{Thm}\label{main}
Assume (A0), (A1), (A2), (A3'), (A4) hold with $K_2=\tilde{K}_2=0$ and $u_0\in \tilde{H}^{0,1}$, then $u_\varepsilon$ satisfies a large deviation principle on the space $L^\infty([0,T], H)\bigcap C([0,T],H^{-1})$ with the good rate function $I^{u_0}$.
\end{Thm}

We aim to prove that $u_\varepsilon$ is exponentially equivalent to the solution to the following equation:
\begin{equation}\label{eq for v}
v_\varepsilon(t)=u_0+\sqrt{\varepsilon}\int^t_0\sigma(\varepsilon s, v_\varepsilon(s))dW(s).
\end{equation}

Because of the non-linear form $b(\cdot,\cdot,\cdot)$ and the anisotropic viscosity, we split the proof into several lemmas. 

\begin{lemma}\label{LDP for v}
Assume $u_0\in \tilde{H}^{0,1}$, then $v_\varepsilon$ satisfies a large deviation principle on the space $L^\infty([0,T], H)\bigcap C([0,T],H^{-1})$ with the good rate function $I^{u_0}$.
\end{lemma}
\begin{proof}
Let $z_\varepsilon$ be the solution to the stochastic equation:
$$z_\varepsilon(t)=u_0+\sqrt{\varepsilon}\int^t_0\sigma(0,z_\varepsilon(s))dW(s).$$

By  \cite[Theorem 12.11]{DZ92}, we know that $z_\varepsilon$ satisfies a large deviation principle with the good rate function $I^{u_0}$. Applying It\^o's formula to $\|v_{\varepsilon}-z_\varepsilon\|^2_H$, we obtain 
\begin{align*}
\|v_{\varepsilon}(t)-z_\varepsilon(t)\|^2_H=&2\sqrt{\varepsilon}\int^t_0\langle v_{\varepsilon}(s)-z_\varepsilon(s),[ \sigma(\varepsilon s, v_{\varepsilon}(s))-\sigma(0,z_\varepsilon(s)) ]dW(s)\rangle\\
&+\varepsilon\int^t_0\|\sigma(\varepsilon s, v_{\varepsilon}(s))-\sigma(0,z_\varepsilon(s))\|^2_{L_2(l^2,H)}ds.
\end{align*}

Then by (A3') and Lemma \ref{martingale lemma}, we get for $p\geqslant 2$,

\begin{align*}
&\left(E[\sup_{0\leqslant t\leqslant T}\|v_{\varepsilon}(t)-z_\varepsilon(t)\|^{2p}_H]\right)^\frac{2}{p}\\
\leqslant &C\varepsilon \left(E[\sup_{0\leqslant t\leqslant T}\int^t_0\langle v_{\varepsilon}(s)-z_\varepsilon(s),(\sigma(\varepsilon s, v_{\varepsilon}(s))-\sigma(0,z_\varepsilon(s)) )dW(s)\rangle ]^{p}\right)^\frac{2}{p}\\
&+C\varepsilon^2 \left(E[\int^T_0\|\sigma(\varepsilon s, v_{\varepsilon}(s))-\sigma(0,z_\varepsilon(s))\|^2_{L_2(l^2,H)}ds ]^p\right)^\frac{2}{p}\\
\leqslant &C\varepsilon p\left(E\left[\int^T_0\|v_{\varepsilon}(s)-z_\varepsilon(s)\|^2_H\|\sigma(\varepsilon s, v_{\varepsilon}(s))-\sigma(0,z_\varepsilon(s))\|^2_{L_2(l^2, H)}ds\right]^\frac{p}{2}\right)^\frac{2}{p}\\
&+C\varepsilon^2\left(\varepsilon^{2\alpha} T^{2+2\alpha}+T\int^T_0\left(E[\sup_{0\leqslant l\leqslant s}\|v_{\varepsilon}(l)-z_\varepsilon(l)\|_H^{2p}]\right)^\frac{2}{p} ds\right)\\
\leqslant&C\varepsilon p\left(\varepsilon^{2\alpha}+\int^T_0\left(E[\sup_{0\leqslant l\leqslant s}\|v_{\varepsilon}(l)-z_\varepsilon(l)\|_H^{2p}]\right)^\frac{2}{p} ds\right)\\
&+C\varepsilon^2\left(\varepsilon^{2\alpha}+\int^T_0\left(E[\sup_{0\leqslant l\leqslant s}\|v_{\varepsilon}(l)-z_\varepsilon(l)\|_H^{2p}]\right)^\frac{2}{p} ds\right).
\end{align*}

By Gronwall's inequality, we have
$$\left(E[\sup_{0\leqslant t\leqslant T}\|v_{\varepsilon}(t)-z_\varepsilon(t)\|^{2p}_H]\right)^{\frac{2}{p}}\leqslant C(\varepsilon^{1+2\alpha} p+\varepsilon^{2+2\alpha})e^{C(\varepsilon p+\varepsilon^2)}.$$

Then Chebyshev's inequality  implies that

\begin{align*}
\varepsilon\log P(\sup_{0\leqslant t\leqslant T}\|v_{\varepsilon}(t)-z_\varepsilon(t)\|^{2}_H>\delta)\leqslant& \varepsilon \log E[\sup_{0\leqslant t\leqslant T}\|v_{\varepsilon}(t)-z_\varepsilon(t)\|^{2p}_H]-\varepsilon p\log \delta\\
\leqslant & \frac{\varepsilon p}{2} (C+C\varepsilon p +C\varepsilon^2+\log (\varepsilon^{1+2\alpha} p+\varepsilon^{2+2\alpha})-2\log \delta).\\
\end{align*}
Let $p=\frac{1}{\varepsilon}$ and $\varepsilon\rightarrow 0$,  we get that  $v_\varepsilon$ and $z_\varepsilon$ are exponentially equivalent, which by Lemma \ref{EXEQ} implies the result.
\end{proof}

\begin{lemma}\label{lemma1}
Let $F_{u_\varepsilon}(t)=\sup_{0\leqslant s\leqslant t}\|u_\varepsilon(s)\|^2_{H}+\varepsilon\int^t_0\|\partial_1 u_\varepsilon(s)\|^2_{H}ds$, then
$$\lim_{M\rightarrow \infty}\sup_{0<\varepsilon\leqslant 1}\varepsilon\log P(F_{u_\varepsilon}(T)>M)=-\infty.$$
\end{lemma}
\begin{proof}

Since $b(u_\varepsilon,u_\varepsilon,u_\varepsilon)=0$, applying It\^o's formula to $\|u_\varepsilon(t)\|^2_H$, we have
\begin{align*}
&\|u_\varepsilon(t)\|^2_H+2\varepsilon\int^t_0\|\partial_1 u_\varepsilon(s)\|^2_{H}ds\\
=&\|u_0\|^2_H+2\sqrt{\varepsilon}\int^t_0\langle u_\varepsilon(s), \sigma(\varepsilon s, u_\varepsilon(s)dW(s)\rangle+\varepsilon\int^t_0\|\sigma(\varepsilon s, u_\varepsilon(s))\|^2_{L_2(l^2,H)}ds.
\end{align*}

Then it follows from (A1) with $K_2=0$ that 
\begin{align*}
\|u_\varepsilon(t)\|^2_H+\varepsilon\int^t_0\|\partial_1 u_\varepsilon(s)\|^2_{H}ds\leqslant& \|u_0\|^2_{\tilde{H}^{0,1}}+C\varepsilon t+ C\varepsilon \int^t_0\|u_\varepsilon(s)\|^2_Hds\\
&+2\sqrt{\varepsilon}\int^t_0\langle u_\varepsilon, \sigma(\varepsilon s, u_\varepsilon(s))dW(s)\rangle.
\end{align*}

Take supremum over $t$, for $p\geqslant 2$,  we have
\begin{align*}
(E[F_{u_\varepsilon}(T)]^p)^\frac{1}{p}\leqslant & \|u_0\|^2_{\tilde{H}^{0,1}}+C\varepsilon T+C\varepsilon \int^T_0 (E[F_{u_\varepsilon}(t)]^p)^\frac{1}{p}dt\\
&+2\sqrt{\varepsilon}(E[\sup_{0\leqslant t\leqslant T}|\int^t_0\langle u_\varepsilon, \sigma(\varepsilon s, u_\varepsilon(s))dW(s)\rangle|]^p)^\frac{1}{p}.
\end{align*}

For the term in the last line, by Lemma \ref{martingale lemma} and \cite[(3.12)]{XZ08}, we have
\begin{align*}
&2\sqrt{\varepsilon}(E[\sup_{0\leqslant t\leqslant T}|\int^t_0\langle u_\varepsilon, \sigma(\varepsilon s, u_\varepsilon(s))dW(s)\rangle|]^p)^\frac{1}{p}\\
\leqslant& C\sqrt{\varepsilon p} \left[\int^T_01+(E\|u_\varepsilon(s)\|^{2p}_H)^\frac{2}{p} ds\right]^\frac{1}{2}.
\end{align*}

Combining the above estimate, we arrive at
\begin{align*}
\left(E[F_{u_\varepsilon}(T)]^p\right)^\frac{2}{p}\leqslant &C\left(\|u_0\|^2_{\tilde{H}^{0,1}}+\varepsilon T\right)^2+C\varepsilon^2\int^T_0\left(E[F_{u_\varepsilon}(t)]^p\right)^\frac{2}{p}ds\\
&+C\varepsilon p T+C\varepsilon p\int^T_0\left(E[F_{u_\varepsilon}(t)]^p\right)^\frac{2}{p}dt.
\end{align*}

Then the Gronwall's inequality implies 
\begin{align*}
\left(E[F_{u_\varepsilon}(T)]^p\right)^\frac{2}{p}\leqslant C\left[\|u_0\|^4_{\tilde{H}^{0,1}}+\varepsilon^2+\varepsilon p  \right]e^{C\varepsilon^2+C\varepsilon p}.
\end{align*}

Let $p=\frac{1}{\varepsilon}$, by Chebyshev's inequality, we have
\begin{align*}
&\varepsilon\log P(F_{u_\varepsilon}(T)>M)\\
\leqslant& -\log M+ \log \left(E[F_{u_\varepsilon}(T)]^p\right)^\frac{1}{p}\\
\leqslant &-\log M+ \log\sqrt{\|u_0\|^4_{\tilde{H}^{0,1}}+\varepsilon^2 +1}+ C(\varepsilon^2+1).
\end{align*}

Take supremum over $\varepsilon$ and let $M\rightarrow \infty$, we finish the proof.

\end{proof}

\begin{lemma}\label{stopping time}
For M>0, define a random time
$$\tau_{M,\varepsilon}=T\wedge\inf\{t:\|u_\varepsilon(t)\|^2_{H}>M,\text{ or }\varepsilon\int^t_0\|\partial_1 u_\varepsilon(s)\|^2_{H}ds>M\}.$$

Then  $\tau_{M,\varepsilon}$ is a stopping time with respect to $\mathcal{F}_{t+}=\cap_{s>t}\mathcal{F}_s$. 

Similarly, Let
$$\tau'_{M,\varepsilon}= T\wedge\inf\{t: \|u_\varepsilon(t)\|^2_{\tilde{H}^{0,1}}>M,\text{ or }\varepsilon\int^t_0\|u_\varepsilon(s)\|^2_{\tilde{H}^{1,1}}ds>M\},$$
then  $\tau'_{M,\varepsilon}$ is a stopping time with respect to $\mathcal{F}_{t+}$.

\end{lemma}

\begin{proof}
The problem comes with the continuity of $u_\varepsilon(t)$. More precisely,  since $\int^t_0\|\partial_1 u_\varepsilon(s)\|^2_{H}ds$ is a continuous adapted process, we only need to prove  that $\hat{\tau}=\inf\{t>0: \|u_\varepsilon(t)\|_H^2>M\}$ is a stopping time. 

 Since $u_\varepsilon\in L^\infty([0,T],H)\bigcap C([0, T], H^{-1})$, 
 $u_\varepsilon(t)$ is weakly continuous on $H$, which implies the lower semi-continuity  of $u_{\varepsilon}$ on $H$.
 
 
By definition of $\hat{\tau}$, for $t>0$
$$\bigcap_{s\in(0,t]}\{\|u_\varepsilon(s)\|^2_H\leqslant M\}\subset \{\hat{\tau}\geqslant t\}\subset \bigcap_{s\in(0,t)}\{\|u_\varepsilon(s)\|^2_H\leqslant M\}.$$
On the contrary, if $\omega\in\{\hat{\tau}\geqslant t\}$, for any $s<t$, $\|u_\varepsilon(s)(\omega)\|^2_H\leqslant M$. Then lower semi-continuity implies $$\|u_\varepsilon(t)(\omega)\|^2_H\leqslant \liminf_{s<t,s\rightarrow t}\|u_\varepsilon(s)\|^2_H\leqslant M.$$ 
Hence we have
$$\{\hat{\tau}\geqslant t\}=\bigcap_{s\in(0,t]}\{\|u_\varepsilon(s)\|^2_H\leqslant M\}.$$

Note that for $\omega\in\bigcap_{s\in(0,t]\cap\mathbb{Q}}\{\|u_\varepsilon(s)\|^2_H\leqslant M\}$, we have for any $s\in(0,t]$, by the lower semi-continuity, $$\|u_\varepsilon(s)(\omega)\|^2_H\leqslant \liminf_{s'\rightarrow s}\|u_\varepsilon(s')\|^2_H\leqslant \liminf_{s'\rightarrow s, s'\in\mathbb{Q}}\|u_\varepsilon(s')\|^2_H\leqslant M,$$
which means 
$$\bigcap_{s\in(0,t]}\{\|u_\varepsilon(s)\|^2_H\leqslant M\}=\bigcap_{s\in(0,t]\cap\mathbb{Q}}\{\|u_\varepsilon(s)\|^2_H\leqslant M\}.$$

 Then we have for $t>0$
 $$\{\hat{\tau}\geqslant t\}=\bigcap_{s\in(0,t]}\{\|u_\varepsilon(s)\|^2_H\leqslant M\}=\bigcap_{s\in(0,t]\cap\mathbb{Q}}\{\|u_\varepsilon(s)\|^2_H\leqslant M\}\in\mathcal{F}_t,$$
 which implies the result.

For $\tau'_{M,\varepsilon}$, the result follows from the fact that $u_\varepsilon$ is weakly continuous in $\tilde{H}^{0,1}$ since $u_\varepsilon\in L^\infty([0,T],\tilde{H}^{0,1})\bigcap C(0,T],H^{-1})$.

\end{proof}

\begin{lemma}\label{lemma2}
Let $G_{u_\varepsilon}(t)=\sup_{0\leqslant s\leqslant t}\|u_\varepsilon(s)\|^2_{\tilde{H}^{0,1}}+\varepsilon\int^{t}_0\|u_\varepsilon(s)\|^2_{\tilde{H}^{1,1}}ds$. 
For fixed $M_1$, we have 
$$\lim_{M\rightarrow \infty}\sup_{0<\varepsilon\leqslant 1}\varepsilon\log P(G_{u_\varepsilon}(\tau_{M_1,\varepsilon})>M)=-\infty.$$
\end{lemma}
\begin{proof}
Let $k$ be a positive constant and $f_\varepsilon(t)=1+\|\partial_1u_\varepsilon(t)\|^2_H$. Applying It\^o's formula to $e^{-k\varepsilon\int^t_0f_\varepsilon(s) ds}\|u_\varepsilon(t)\|^2_{\tilde{H}^{0,1}}$, we obtain

\begin{align*}
&e^{-k\varepsilon\int^t_0f_\varepsilon(s) ds}\|u_\varepsilon(t)\|^2_{\tilde{H}^{0,1}}+2\varepsilon\int^t_0e^{-k\varepsilon\int^s_0f_\varepsilon(r)dr}(\|\partial_1u_\varepsilon(s)\|^2_{H}+\|\partial_1\partial_2u_\varepsilon(s)\|^2_H)ds\\
=&\|u_0\|^2_{\tilde{H}^{0,1}}-k\varepsilon\int^t_0e^{-k\varepsilon\int^s_0f_\varepsilon(r)dr}f_\varepsilon(s)\| u_\varepsilon(s)\|^2_{\tilde{H}^{0,1}}ds\\
&-2\varepsilon\int^t_0e^{-k\varepsilon\int^s_0f_\varepsilon(r)dr}\langle\partial_2u_\varepsilon(s), \partial_2(u_\varepsilon\cdot\nabla u_\varepsilon)(s)\rangle ds\\
&+2\sqrt{\varepsilon}\int^t_0e^{-k\varepsilon\int^s_0f_\varepsilon(r)dr}\langle u_\varepsilon(s), \sigma(\varepsilon s, u_\varepsilon(s))dW(s)\rangle_{\tilde{H}^{0,1}}\\
&+\varepsilon\int^t_0e^{-k\varepsilon\int^s_0f_\varepsilon(r)dr}\|\sigma(\varepsilon s, u_\varepsilon(s))\|^2_{L_2(l^2, {\tilde{H}^{0,1}})}ds.
\end{align*}

The fourth and the fifth line can be dealt in the same way as in the proof of Lemma \ref{lemma1}.  For the third line, by Lemma \ref{estimate for b with partial_2}, we have
\begin{align*}
&|\langle\partial_2u_\varepsilon, \partial_2(u_\varepsilon\cdot\nabla u_\varepsilon)\rangle|
\leqslant \frac{1}{2}\|\partial_1\partial_2 u_\varepsilon\|_H^2+ C_1f_\varepsilon\|\partial_2u_\varepsilon\|_H^2,
\end{align*}
where $C_1$ is a constant.
Therefore by (A2) with $\tilde{K}_2=0$ we get
\begin{align*}
&e^{-k\varepsilon\int^t_0f_\varepsilon(s) ds}\|u_\varepsilon(t)\|^2_{\tilde{H}^{0,1}}+\varepsilon\int^t_0e^{-k\varepsilon\int^s_0f_\varepsilon(r)dr}\|u_\varepsilon(s)\|^2_{\tilde{H}^{1,1}}ds\\
\leqslant&\|u_0\|^2_{\tilde{H}^{0,1}}-k\varepsilon\int^t_0e^{-k\varepsilon\int^s_0f_\varepsilon(r)dr}f_\varepsilon(s)\|u_\varepsilon(s)\|^2_{\tilde{H}^{0,1}}ds\\
&+2C_1\varepsilon\int^t_0e^{-k\varepsilon\int^s_0f_\varepsilon(r)dr}f_\varepsilon(s)\|u_\varepsilon(s)\|^2_{\tilde{H}^{0,1}} ds\\
&+2\sqrt{\varepsilon}\int^t_0e^{-k\varepsilon\int^s_0f_\varepsilon(r)dr}\langle u_\varepsilon(s), \sigma(\varepsilon s, u_\varepsilon(s))dW(s)\rangle_{\tilde{H}^{0,1}}\\
&+\varepsilon\int^t_0e^{-k\varepsilon\int^s_0f_\varepsilon(r)dr}[\tilde{K_0}+(\tilde{K}_1+1)\|u_\varepsilon(s)\|^2_{\tilde{H}^{0,1}}]ds.
\end{align*}

For the last second line, similar to \cite[(3.12)]{XZ08}, we have

\begin{align*}
&2\sqrt{\varepsilon}(E[\sup_{0\leqslant s\leqslant t}|\int^s_0e^{-k\varepsilon\int^r_0f_\varepsilon(l)dl}\langle u_\varepsilon(r), \sigma(\varepsilon r, u_\varepsilon(r))dW(r)\rangle_{\tilde{H}^{0,1}}|]^p)^\frac{1}{p}\\
\leqslant& C\sqrt{\varepsilon p} (E[\int^t_0e^{-2k\varepsilon\int^r_0f_\varepsilon(l)dl}\| u_\varepsilon(r)\|_{\tilde{H}^{0,1}}^2\|\sigma(\varepsilon r, u_\varepsilon(r))\|^2_{L_2(l^2,\tilde{H}^{0,1})}dr]^\frac{p}{2})^\frac{1}{p}\\
\leqslant& C\sqrt{\varepsilon p} (E[\int^t_0e^{-2k\varepsilon\int^r_0f_\varepsilon(l)dl}\| u_\varepsilon(r)\|_{\tilde{H}^{0,1}}^2(1+\|u_\varepsilon(r)\|^2_{\tilde{H}^{0,1}})dr]^\frac{p}{2})^\frac{1}{p}\\
\leqslant& C\sqrt{\varepsilon p} (E[\int^t_0e^{-2k\varepsilon\int^r_0f_\varepsilon(l)dl}(1+\|u_\varepsilon(r)\|^4_{\tilde{H}^{0,1}})dr]^\frac{p}{2})^\frac{1}{p}\\
\leqslant& C\sqrt{\varepsilon p} \left[\int^t_01+(E[e^{-pk\varepsilon\int^r_0f_\varepsilon(l)dl}\|u_\varepsilon(s)\|^{2p}_{\tilde{H}^{0,1}}])^\frac{2}{p} ds\right]^\frac{1}{2},
\end{align*}
where we used (A2) with $K_2=0$ in the third line.

Let $k>2C_1$ and using Lemma \ref{martingale lemma}, we have for $p\geqslant 2$
\begin{align*}
&\left(E\left[\sup_{0\leqslant s\leqslant t\wedge\tau_{M_1,\varepsilon}}e^{-k\varepsilon\int^s_0f_\varepsilon(r)dr}\|u_\varepsilon(s)\|^2_{\tilde{H}^{0,1}}+\varepsilon\int^{t\wedge\tau_{M_1,\varepsilon}}_0e^{-k\varepsilon\int^s_0f_\varepsilon(r)dr}\|u_\varepsilon(s)\|^2_{\tilde{H}^{1,1}}ds\right]^p\right)^\frac{2}{p}\\
\leqslant&C(\|u_0\|^2_{\tilde{H}^{0,1}}+\varepsilon)^2+C\varepsilon^2\int^t_0\left(E\left[\sup_{0\leqslant r\leqslant s\wedge\tau_{M_1,\varepsilon}}e^{-k\varepsilon\int^r_0f_\varepsilon(l)dl}\|u_\varepsilon(r)\|^2_{\tilde{H}^{0,1}}\right]^p\right)^\frac{2}{p}ds\\
&+C\varepsilon p+C\varepsilon p\int^t_0\left(E\left[\sup_{0\leqslant r\leqslant s\wedge\tau_{M_1,\varepsilon}}e^{-k\varepsilon\int^r_0f_\varepsilon(l)dl}\|u_\varepsilon(r)\|^2_{\tilde{H}^{0,1}}\right]^p\right)^\frac{2}{p}ds.
\end{align*}

Applying Gronwall's inequality, we obtain
\begin{align*}
&\left(E\left[\sup_{0\leqslant t\leqslant \tau_{M_1,\varepsilon}}e^{-k\varepsilon\int^t_0f_\varepsilon(s)ds}\|u_\varepsilon(t)\|^2_{\tilde{H}^{0,1}}+\varepsilon\int^{\tau_{M_1,\varepsilon}}_0e^{-k\varepsilon\int^s_0f_\varepsilon(r)dr}\|u_\varepsilon(s)\|^2_{\tilde{H}^{1,1}}ds\right]^p\right)^\frac{2}{p}\\
\leqslant &C\left[\|u_0\|^4_{\tilde{H}^{0,1}}+\varepsilon^2+\varepsilon p\right]e^{C(\varepsilon^2+\varepsilon p)}.
\end{align*}

Hence by the definition of $\tau_{M_1,\varepsilon}$, we have
\begin{align*}
&\left(E\left[G_{u_\varepsilon}(\tau_{M_1,\varepsilon})\right]^p\right)^\frac{2}{p}\\
\leqslant& \left(E\left[\left(\sup_{0\leqslant t\leqslant \tau_{M_1,\varepsilon}}e^{-k\varepsilon\int^t_0f_\varepsilon(s)ds}\|u_\varepsilon(t)\|^2_{\tilde{H}^{0,1}}+\varepsilon\int^{\tau_{M_1,\varepsilon}}_0e^{-k\varepsilon\int^s_0f_\varepsilon(r)dr}\|u_\varepsilon(s)\|^2_{\tilde{H}^{1,1}}ds\right)^pe^{pk\varepsilon\int^t_0f_\varepsilon(s)ds}\right]\right)^\frac{2}{p}\\
\leqslant&e^{C(M_1+\varepsilon)}\left(E\left[\sup_{0\leqslant t\leqslant \tau_{M_1,\varepsilon}}e^{-k\varepsilon\int^t_0f_\varepsilon(s)ds}\|u_\varepsilon(t)\|^2_{\tilde{H}^{0,1}}+\varepsilon\int^{\tau_{M_1,\varepsilon}}_0e^{-k\varepsilon\int^s_0f_\varepsilon(r)dr}\|u_\varepsilon(s)\|^2_{\tilde{H}^{1,1}}ds\right]^p\right)^\frac{2}{p}\\
\leqslant &Ce^{C(M_1+\varepsilon)}\left[\|u_0\|^4_{\tilde{H}^{0,1}}+\varepsilon^2+\varepsilon p\right]e^{C(\varepsilon^2+\varepsilon p)}.
\end{align*}

Let $p=\frac{2}{\varepsilon}$, by Chebyshev's inequality, we have
\begin{align*}
&\varepsilon\log P(G_{u_\varepsilon}(\tau_{M_1,\varepsilon})>M)\\
\leqslant &\varepsilon\log\frac{E\left[G_{u_\varepsilon}(\tau_{M_1,\varepsilon})\right]^p}{M^p}\\
\leqslant&-2\log M+C+C(M_1+\varepsilon)+C(\varepsilon^2+\varepsilon p)+\log[\|u_0\|^4_{\tilde{H}^{0,1}}+\varepsilon^2+\varepsilon p].
\end{align*}

Take supremum over $\varepsilon$ and let $M\rightarrow \infty$, we finish the proof.
\end{proof}

Since $V$ is dense in ${\tilde{H}^{0,1}}$, there exists a sequence $\{u^n_0\}\subset V$ such that
$$\lim_{n\rightarrow +\infty}\|u^n_0-u_0\|_{\tilde{H}^{0,1}}=0.$$

Let $u_{n,\varepsilon}$ be the solution to (\ref{eq. for small time}) with  the initial data $u^n_{0}$. Similarly, let $v_{n,\varepsilon}$ be the solution to (\ref{eq for v}) with  the initial data $u^n_{0}$.

For $M>0$, define a random time (which is also a stopping time with respect to $\mathcal{F}_{t+}$ by Lemma \ref{stopping time})
$$\tau^n_{M,\varepsilon}:=T\wedge\inf\{t:\|u_{n,\varepsilon}(t)\|^2_{H}>M,\text{ or }\varepsilon\int^t_0\|\partial_1 u_{n,\varepsilon}(s)\|^2_{H}ds>M\}.$$

From the proof of Lemma \ref{lemma1} and Lemma \ref{lemma2}, it follows that 

\begin{lemma}\label{lemma3}
$$\lim_{M\rightarrow \infty}\sup_n\sup_{0<\varepsilon\leqslant 1}\varepsilon\log P(F_{u_{n,\varepsilon}}(T)>M)=-\infty.$$

For fixed $M_1$, we have 
$$\lim_{M\rightarrow \infty}\sup_n\sup_{0<\varepsilon\leqslant 1}\varepsilon\log P(G_{u_{n,\varepsilon}}(\tau_{M_1,\varepsilon}^n)>M)=-\infty.$$
\end{lemma}

 \vskip.10in

The following lemma for $v_{n,\varepsilon}$ is from  \cite{XZ08}:

\begin{lemma}[{\cite[Lemma 3.2]{XZ08}}]\label{lemma4}
$$\lim_{M\rightarrow\infty}\sup_{0<\varepsilon\leqslant 1}\varepsilon\log P\left(\sup_{0\leqslant t\leqslant T}\|v_{n,\varepsilon}(t)\|^2_{V}>M\right)=-\infty.$$
\end{lemma}
 \vskip.10in
\begin{lemma}\label{lemma5}
For any $\delta>0$, 
$$\lim_{n\rightarrow\infty}\sup_{0<\varepsilon\leqslant 1}\varepsilon\log P\left(\sup_{0\leqslant t\leqslant T}\|u_{n,\varepsilon}(t)-u_\varepsilon(t)\|_H^2>\delta\right)=-\infty.$$
\end{lemma}
\begin{proof}

Clearly, for $M_1, M_2>0$
\begin{equation}\label{eq in lemma5}\aligned
&P\left(\sup_{0\leqslant t\leqslant T}\|u_{n,\varepsilon}(t)-u_\varepsilon(t)\|_H^2>\delta\right)\\
\leqslant &P\left(\sup_{0\leqslant t\leqslant T}\|u_{n,\varepsilon}(t)-u_\varepsilon(t)\|_H^2>\delta, F_{u_\varepsilon}(T)\leqslant M_1, G_{u_\varepsilon}(T)\leqslant M_2\right)\\
&+P\left(F_{u_\varepsilon}(T)>M_1\right)+P\left(F_{u_\varepsilon}(T)\leqslant M_1, G_{u_\varepsilon}(T)>M_2\right)\\
\leqslant &P\left(\sup_{0\leqslant t\leqslant\tau_{M_1,\varepsilon}\wedge\tau'_{M_2,\varepsilon}}\|u_{n,\varepsilon}(t)-u_\varepsilon(t)\|_H^2>\delta\right)\\
&+P\left(F_{u_\varepsilon}(T)>M_1\right)+P\left(G_{u_\varepsilon}(\tau_{M_1,\varepsilon})>M_2\right),
\endaligned
\end{equation}
where $\tau_{M_1,\varepsilon}$ and $\tau'_{M_2,\varepsilon}$ are introduced in Lemma \ref{stopping time}.

 For the first term on the right hand of (\ref{eq in lemma5}), let $k$ be a positive constant and 
\begin{align*}
U_\varepsilon= 1+\|u_\varepsilon\|^2_{\tilde{H}^{1,1}}.
\end{align*}
 Applying It\^o's formula to $e^{-\varepsilon k\int^{t}_0U_\varepsilon(s) ds}\|u_\varepsilon(t)-u_{n,\varepsilon}(t)\|_H^2$, we get
\begin{align*}
&e^{-\varepsilon k\int^{t}_0U_\varepsilon(s) ds}\|u_\varepsilon(t)-u_{n,\varepsilon}(t)\|_H^2+2\varepsilon\int^t_0e^{-\varepsilon k\int^s_0U_\varepsilon(r) dr}\|\partial_1(u_\varepsilon(s)-u_{n,\varepsilon}(s))\|^2_Hds\\
=&\|u_0-u_{n,0}\|^2_H-k\varepsilon\int^t_0e^{-\varepsilon k\int^s_0U_\varepsilon(r) dr}U_\varepsilon(s)\|u_\varepsilon(s)-u_{n,\varepsilon}(s)\|^2_Hds\\
&-2\varepsilon\int^t_0e^{-\varepsilon k\int^s_0U_\varepsilon(r) dr}\left(b(u_\varepsilon,u_\varepsilon,u_\varepsilon-u_{n,\varepsilon})(s)-b(u_{n,\varepsilon},u_{n,\varepsilon},u_\varepsilon-u_{n,\varepsilon})(s)\right)ds\\
&+\varepsilon\int^t_0e^{-\varepsilon k\int^s_0U_\varepsilon(r) dr}\|\sigma(\varepsilon s, u_\varepsilon(s))-\sigma(\varepsilon s, u_{n,\varepsilon}(s))\|^2_{L_2(l^2, H)}ds\\
&+2\sqrt{\varepsilon}\int^t_0e^{-\varepsilon k\int^s_0U_\varepsilon(r) dr}\langle u_\varepsilon(s)-u_{n,\varepsilon}(s), (\sigma(\varepsilon s, u_\varepsilon(s))-\sigma(\varepsilon s, u_{n,\varepsilon}(s)))dW(s)\rangle.
\end{align*}

Notice that by the property of the trilinear form $b$ and Lemma \ref{anisotropic estimate for b}, we have
\begin{align*}
&|b(u_\varepsilon,u_\varepsilon,u_\varepsilon-u_{n,\varepsilon})-b(u_{n,\varepsilon},u_{n,\varepsilon},u_\varepsilon-u_{n,\varepsilon})|\\
=&|b(u_\varepsilon,u_\varepsilon,u_\varepsilon-u_{n,\varepsilon})-b(u_{n,\varepsilon},u_\varepsilon,u_\varepsilon-u_{n,\varepsilon})|\\
=&|b(u_\varepsilon-u_{n,\varepsilon},u_\varepsilon,u_\varepsilon-u_{n,\varepsilon})|\\
\leqslant& \frac{1}{2}\|\partial_1(u_\varepsilon-u_{n,\varepsilon})\|^2_H+C_1U_\varepsilon\|u_\varepsilon-u_{n,\varepsilon}\|^2_H,
\end{align*}
where $C_1$ is a constant.

Therefore, 
\begin{align*}
&e^{-\varepsilon k\int^{t}_0U_\varepsilon(s) ds}\|u_\varepsilon(t)-u_{n,\varepsilon}(t)\|_H^2\\
\leqslant& \|u_0-u_{n,0}\|^2_{\tilde{H}^{0,1}}-k\varepsilon\int^t_0e^{-\varepsilon k\int^s_0U_\varepsilon(r) dr}U_\varepsilon(s)\|u_\varepsilon(s)-u_{n,\varepsilon}(s)\|^2_Hds\\
&+2\varepsilon C_1 \int^t_0e^{-\varepsilon k\int^s_0U_\varepsilon(r) dr}U_\varepsilon(s)\|u_\varepsilon(s)-u_{n,\varepsilon}(s)\|^2_Hds\\
&+L\varepsilon\int^t_0e^{-\varepsilon k\int^s_0U_\varepsilon(r) dr}\|u_\varepsilon(s)-u_{n,\varepsilon}(s)\|^2_{H}ds\\
&+2\sqrt{\varepsilon}\int^t_0e^{-\varepsilon k\int^s_0U_\varepsilon(r) dr}\langle u_\varepsilon(s)-u_{n,\varepsilon}(s), (\sigma(\varepsilon s, u_\varepsilon(s))-\sigma(\varepsilon s, u_{n,\varepsilon}(s)))dW(s)\rangle,
\end{align*}
where we used (A3') in the forth line.

Choosing $k>2C_1$ and using Lemma \ref{martingale lemma} and (A3'), by the similar calculation as in the proof of Lemma \ref{lemma2} we have for $p\geqslant 2$
\begin{align*}
&\left(E\left[\sup_{0\leqslant s\leqslant t\wedge\tau_{M_1,\varepsilon}\wedge\tau'_{M_2,\varepsilon}}e^{-\varepsilon k\int^s_0U_\varepsilon(r) dr}\|u_\varepsilon(s)-u_{n,\varepsilon}(s)\|^2_{H}\right]^p\right)^\frac{2}{p}\\
\leqslant& 2\|u_0-u_{n,0}\|^4_{\tilde{H}^{0,1}}+C\varepsilon^2\int^t_0\left(E\left[\sup_{0\leqslant r\leqslant s\wedge\tau_{M_1,\varepsilon}\wedge\tau'_{M_2,\varepsilon}}e^{-\varepsilon k\int^r_0U_\varepsilon(l) dl}\|u_\varepsilon(r)-u_{n,\varepsilon}(r)\|^2_{H}\right]^p\right)^\frac{2}{p}ds\\
&+C\varepsilon p\int^t_0\left(E\left[\sup_{0\leqslant r\leqslant s\wedge\tau_{M_1,\varepsilon}\wedge\tau'_{M_2,\varepsilon}}e^{-\varepsilon k\int^r_0U_\varepsilon(l) dl}\|u_\varepsilon(r)-u_{n,\varepsilon}(r)\|^2_{H}\right]^p\right)^\frac{2}{p}ds.\\
\end{align*}

Applying Gronwall's inequality, we obtain
\begin{align*}
&\left(E\left[\sup_{0\leqslant s\leqslant t\wedge\tau_{M_1,\varepsilon}\wedge\tau'_{M_2,\varepsilon}}e^{-\varepsilon k\int^s_0U_\varepsilon(r) dr}\|u_\varepsilon(s)-u_{n,\varepsilon}(s)\|^2_{H}\right]^p\right)^\frac{2}{p}\leqslant C\|u_0-u_{n,0}\|^4_{\tilde{H}^{0,1}}e^{C(\varepsilon^2+\varepsilon p)}.
\end{align*}
Hence, by the definition of the stopping times, 
\begin{align*}
&\left(E\left[\sup_{0\leqslant s\leqslant \tau_{M_1,\varepsilon}\wedge\tau'_{M_2,\varepsilon}}\|u_\varepsilon(s)-u_{n,\varepsilon}(s)\|^2_{H}\right]^p\right)^\frac{2}{p}\\
\leqslant&\left(E\left[\big{(}\sup_{0\leqslant s\leqslant \tau_{M_1,\varepsilon}\wedge\tau'_{M_2,\varepsilon}}e^{-\varepsilon k\int^s_0U_\varepsilon(r) dr}\|u_\varepsilon(s)-u_{n,\varepsilon}(s)\|^2_{H}\big{)}^pe^{kp\varepsilon\int^{\tau_{M_1,\varepsilon}\wedge\tau'_{M_2,\varepsilon}}_0U_\varepsilon(s)ds}\right]\right)^\frac{2}{p}\\
\leqslant&e^{C(\varepsilon+M_2)k}\left(E\left[\sup_{0\leqslant s\leqslant \tau_{M_1,\varepsilon}\wedge\tau'_{M_2,\varepsilon}}e^{-\varepsilon k\int^s_0U_\varepsilon(r) dr}\|u_\varepsilon(s)-u_{n,\varepsilon}(s)\|^2_{H}\right]^p\right)^\frac{2}{p}\\
\leqslant&Ce^{C(\varepsilon+M_2)k}\|u_0-u_{n,0}\|^4_{\tilde{H}^{0,1}}e^{C(\varepsilon^2+\varepsilon p)}.
\end{align*}

Fix $M_1, M_2$, let $p=\frac{2}{\varepsilon}$, then Chebyshev's inequality implies that
\begin{align*}
&\sup_{0<\varepsilon\leqslant 1}\varepsilon\log P\left(\sup_{0\leqslant t\leqslant \tau_{M_1,\varepsilon}\wedge\tau'_{M_2,\varepsilon}}\|u_{n,\varepsilon}(t)-u_\varepsilon(t)\|_H^2>\delta\right)\\
\leqslant&\sup_{0<\varepsilon\leqslant 1}\varepsilon\log\frac{E\left[\sup_{0\leqslant t\leqslant \tau_{M_1,\varepsilon}\wedge\tau'_{M_2,\varepsilon}}\|u_{n,\varepsilon}(t)-u_\varepsilon(t)\|_H^{2p}\right]}{\delta^p}\\
\leqslant& C(\varepsilon+M_2)-2\log\delta+\log\|u_0-u_{n,0}\|^4_{\tilde{H}^{0,1}}+C(\varepsilon^2+\varepsilon p)+C\\
\rightarrow&-\infty, \text{        as }n\rightarrow \infty.
\end{align*}

By Lemma \ref{lemma1}, for any $R>0$, there exists a constant $M_1$ such that for any $\varepsilon\in(0,1]$, 
$$P\left(F_{u_\varepsilon}(T)>M_1\right)\leqslant e^{-\frac{R}{\varepsilon}}.$$

For such a $M_1$, by Lemma \ref{lemma2}, there exists a constant $M_2$ such that for any $\varepsilon\in(0,1]$,
$$P\left(G_{u_\varepsilon}(\tau_{M_1,\varepsilon})>M_2\right)\leqslant e^{-\frac{R}{\varepsilon}}.$$

For such $M_1,M_2$, there exists a positive integer $N$, such that for any $n\geqslant N$ and $\varepsilon\in(0,1]$,
$$P\left(\sup_{0\leqslant t\leqslant\tau_{M_1,\varepsilon}\wedge\tau'_{M_2,\varepsilon}}\|u_{n,\varepsilon}(t)-u_\varepsilon(t)\|_H^2>\delta\right)\leqslant e^{-\frac{R}{\varepsilon}}.$$

Then by (\ref{eq in lemma5}), we see that there exists a positive integer $N$, such that for any $n\geqslant N$, $\varepsilon\in(0,1]$,
$$P\left(\sup_{0\leqslant t\leqslant T}\|u_{n,\varepsilon}(t)-u_\varepsilon(t)\|_H^2>\delta\right)\leqslant 3e^{-\frac{R}{\varepsilon}}.$$

Since $R$ is arbitrary, the lemma follows.
\end{proof}

The following  lemma for $v_\varepsilon$ is from \cite{XZ08}:

\begin{lemma}[{\cite[Lemma 3.4]{XZ08}}]\label{lemma6}
For any $\delta>0$, 
$$\lim_{n\rightarrow\infty}\sup_{0<\varepsilon\leqslant 1}\varepsilon\log P\left(\sup_{0\leqslant t\leqslant T}\|v_{n,\varepsilon}(t)-v_\varepsilon(t)\|_H^2>\delta\right)=-\infty.$$
\end{lemma}
 \vskip.10in

\begin{lemma}\label{lemma7}
For any $\delta>0$, and every positive integer $n$,
$$\lim_{\varepsilon\rightarrow 0}\varepsilon\log P\left(\sup_{0\leqslant t\leqslant T}\|u_{n,\varepsilon}(t)-v_{n,\varepsilon}(t)\|^2_H>\delta\right)=-\infty.$$
\end{lemma}
\begin{proof}
For $M>0$, recall the definition of $\tau^n_{M, \varepsilon}$  and define the following random time:
$$\tau^{2,n}_{M,\varepsilon}:= T\wedge\inf\{t: \|u_{n,\varepsilon}(t)\|^2_{\tilde{H}^{0,1}}>M,\text{ or }\varepsilon\int^t_0\|u_{n,\varepsilon}(s)\|^2_{\tilde{H}^{1,1}}ds>M\},$$
which is a stopping time with respect to $\mathcal{F}_{t+}$ by Lemma \ref{stopping time}.

Moreover, define
$$\tau^{3,n}_{M,\varepsilon}:= T\wedge\inf\{t: \|v_{n,\varepsilon}(t)\|^2_{V}>M\},$$
$$\tau^{1,n}_{M,\varepsilon}:=\tau^n_{M,\varepsilon}\wedge\tau^{3,n}_{M,\varepsilon}.$$
We should point out that $\tau^{3,n}_{M,\varepsilon}$ is a stopping time with respect to $\mathcal{F}_t$ under the condition $v_{n,\varepsilon}\in C([0,T], V)$. Now we prove that $v_{n,\varepsilon}\in C([0,T], V)$.

By It\^o's formula and  Gronwall's inequality there exists a constant $C(\varepsilon)$ such that
\begin{align*}
E(\sup_{s\in[0,t]}\|v_{n,\varepsilon}(s)\|_V^2)\leqslant C(\varepsilon).
\end{align*}
For $0\leqslant s<t\leqslant T$, by (A4) we have
\begin{align*}
E\|v_{n,\varepsilon}(t)-v_{n,\varepsilon}(s)\|_V^2\leqslant& \varepsilon E\int^t_s\|\sigma(\varepsilon r, v_{n,\varepsilon}(r))\|^2_{L_2(l^2,V)}dr\\
\leqslant &\varepsilon \int^t_s(\overline{K}_0+\overline{K}_1E(\sup_{l\in[0,r]}\|v_{n,\varepsilon}(l)\|_V^2))dr\\
\leqslant& \varepsilon(\overline{K}_0+\overline{K}_1C(\varepsilon))|t-s|.
\end{align*}
Then Kolmogorov's continuity criterion implies that $v_{n,\varepsilon}\in C([0,T], V)$.

Now for $M_1, M_2>0$, similarly to (\ref{eq in lemma5}), we have 
\begin{equation}\label{eq in lemma7}\aligned
&P\left(\sup_{0\leqslant t\leqslant T}\|u_{n,\varepsilon}(t)-v_{n,\varepsilon}(t)\|^2_H>\delta\right)\\\leqslant&P\left(\sup_{0\leqslant t\leqslant \tau^{1,n}_{M_1,\varepsilon}\wedge\tau^{2,n}_{M_2,\varepsilon}}\|u_{n,\varepsilon}(t)-v_{n,\varepsilon}(t)\|^2_H>\delta\right)\\
&+P(F_{u_{n,\varepsilon}}(T)>M_1)+P(G_{u_{n,\varepsilon}}(\tau^n_{M_1,\varepsilon})>M_2)+P\left(\sup_{0\leqslant t\leqslant T}\|v_{n,\varepsilon}(t)\|^2_{V}>M_1\right)
\endaligned
\end{equation}

Let $U_{n,\varepsilon}=1+\|u_{n,\varepsilon}\|^2_{\tilde{H}^{1,1}}$, 
 applying It\^o's formula to $e^{-k\varepsilon\int^t_0U_{n,\varepsilon}(s)ds}\|u_{n,\varepsilon}(t)-v_{n,\varepsilon}(t)\|^2_H$ for some constant $k>0$, we get
\begin{equation}\label{eq2 in lemma7}\aligned
&e^{-k\varepsilon\int^t_0U_{n,\varepsilon}(s)ds}\|u_{n,\varepsilon}(t)-v_{n,\varepsilon}(t)\|^2_H+2\varepsilon\int^t_0e^{-k\varepsilon\int^s_0U_{n,\varepsilon}(r)dr}\|\partial_1(u_{n,\varepsilon}(s)-v_{n,\varepsilon}(s))\|^2_Hds\\
=&-k\varepsilon\int^t_0e^{-k\varepsilon\int^s_0U_{n,\varepsilon}(r)dr}U_{n,\varepsilon}(s)\|u_{n,\varepsilon}(s)-v_{n,\varepsilon}(s)\|^2_Hds\\
&+2\varepsilon\int^t_0e^{-k\varepsilon\int^s_0U_{n,\varepsilon}(r)dr}\langle u_{n,\varepsilon}(s)-v_{n,\varepsilon}(s), \partial_1^2 v_{n,\varepsilon}(s)\rangle ds\\
&-2\varepsilon\int^t_0e^{-k\varepsilon\int^s_0U_{n,\varepsilon}(r)dr}b(u_{n,\varepsilon}(s),u_{n,\varepsilon}(s),u_{n,\varepsilon}(s)-v_{n,\varepsilon}(s))ds\\
&+\varepsilon\int^t_0e^{-k\varepsilon\int^s_0U_{n,\varepsilon}(r)dr}\|\sigma(\varepsilon s,u_{n,\varepsilon}(s))-\sigma(\varepsilon s, v_{n,\varepsilon}(s))\|^2_{L_2(l^2,H)}ds\\
&+2\sqrt{\varepsilon}\int^t_0e^{-k\varepsilon\int^s_0U_{n,\varepsilon}(r)dr}\langle u_{n,\varepsilon}(s)-v_{n,\varepsilon}(s), (\sigma(\varepsilon s,u_{n,\varepsilon}(s))-\sigma(\varepsilon s, v_{n,\varepsilon}(s)))dW(s) \rangle.
\endaligned
\end{equation}

For the second term on the right hand side of (\ref{eq2 in lemma7}), we have
\begin{align*}
&\Big{|}\int^t_0e^{-k\varepsilon\int^s_0U_{n,\varepsilon}(r)dr}\langle u_{n,\varepsilon}(s)-v_{n,\varepsilon}(s), \partial_1^2 v_{n,\varepsilon}(s)\rangle ds\Big{|}\\
\leqslant& \int^t_0e^{-k\varepsilon\int^s_0U_{n,\varepsilon}(r)dr}\|\partial_1( u_{n,\varepsilon}(s)-v_{n,\varepsilon}(s))\|_H\| \partial_1 v_{n,\varepsilon}(s)\|_H ds\\
\leqslant& \frac{1}{4}\int^t_0e^{-k\varepsilon\int^s_0U_{n,\varepsilon}(r)dr}\| \partial_1(u_{n,\varepsilon}(s)-v_{n,\varepsilon}(s))\|^2_Hds+C\int^t_0e^{-k\varepsilon\int^s_0U_{n,\varepsilon}(r)dr}\| v_{n,\varepsilon}(s)\|^2_{V} ds,
\end{align*}
where we use Young's inequality in the last inequality. 

For the third term on the right hand side of \eqref{eq2 in lemma7}, by Lemmas \ref{anisotropic estimate for b} and \ref{b(u,v,w)} we have
\begin{equation}\aligned
&|b(u_{n,\varepsilon},u_{n,\varepsilon},u_{n,\varepsilon}-v_{n,\varepsilon})|\\
=&|b(u_{n,\varepsilon}-v_{n,\varepsilon},u_{n,\varepsilon},u_{n,\varepsilon}-v_{n,\varepsilon})+b(v_{n,\varepsilon},u_{n,\varepsilon},u_{n,\varepsilon}-v_{n,\varepsilon}|\\
\leqslant&\frac{1}{4}\|\partial_1(u_{n,\varepsilon}-v_{n,\varepsilon})\|^2_H+CU_{n,\varepsilon}\|u_{n,\varepsilon}-v_{n,\varepsilon}\|^2_H+ C\|v_{n,\varepsilon}\|_{V}\|u_{n,\varepsilon}\|_{\tilde{H}^{1,1}}\|u_{n,\varepsilon}-v_{n.\varepsilon}\|_H\\
\leqslant&\frac{1}{4}\|\partial_1(u_{n,\varepsilon}-v_{n,\varepsilon})\|^2_H+C\|v_{n,\varepsilon}\|^2_{V}+C_1 U_{n,\varepsilon}\|u_{n,\varepsilon}-v_{n,\varepsilon}\|^2_H,
\endaligned
\end{equation}
 where $C_1$ is a constant.

Thus we obtain
\begin{align*}
&e^{-k\varepsilon\int^t_0U_{n,\varepsilon}(s)ds}\|u_{n,\varepsilon}(t)-v_{n,\varepsilon}(t)\|^2_H+\varepsilon\int^t_0e^{-k\varepsilon\int^s_0U_{n,\varepsilon}(r)dr}\|\partial_1(u_{n,\varepsilon}(s)-v_{n,\varepsilon}(s))\|^2_Hds\\
\leqslant&-k\varepsilon\int^t_0e^{-k\varepsilon\int^s_0U_{n,\varepsilon}(r)dr}U_{n,\varepsilon}(s)\|u_{n,\varepsilon}(s)-v_{n,\varepsilon}(s)\|^2_Hds+C\varepsilon \int^t_0e^{-k\varepsilon\int^s_0U_{n,\varepsilon}(r)dr}\| v_{n,\varepsilon}(s)\|^2_{V} ds\\
&+C_1\varepsilon \int^t_0e^{-k\varepsilon\int^s_0U_{n,\varepsilon}(r)dr}U_{n,\varepsilon}(s)\|u_{n,\varepsilon}(s)-v_{n,\varepsilon}(s)\|^2_Hds\\
&+L_1\varepsilon \int^t_0e^{-k\varepsilon\int^s_0U_{n,\varepsilon}(r)dr}\|u_{n,\varepsilon}(s)-v_{n,\varepsilon}(s)\|^2_Hds\\
&+2\sqrt{\varepsilon}\int^t_0e^{-k\varepsilon\int^s_0U_{n,\varepsilon}(r)dr}\langle u_{n,\varepsilon}(s)-v_{n,\varepsilon}(s), (\sigma(\varepsilon s,u_{n,\varepsilon}(s))-\sigma(\varepsilon s, v_{n,\varepsilon}(s)))dW(s) \rangle,
\end{align*}
where  we used  (A3') in the fourth line.

Hence, choosing $k>C_1+C_2$, by Lemma \ref{martingale lemma} and the similar techniques in the previous lemma and the definition of stopping times, we deduce that for $p\geqslant 2$
\begin{align*}
&\left(E\left[\sup_{0\leqslant s\leqslant t\wedge\tau^{1,n}_{M_1,\varepsilon}\wedge\tau^{2,n}_{M_2,\varepsilon}}e^{-k\varepsilon\int^s_0U_{n,\varepsilon}(r)dr}\|u_{n,\varepsilon}(s)-v_{n,\varepsilon}(s)\|^2_H\right]^p\right)^{\frac{2}{p}}\\
\leqslant& CM_1^2\varepsilon^2+C(\varepsilon^2+\varepsilon p)\int^t_0\left(E\left[\sup_{0\leqslant r\leqslant s\wedge\tau^{1,n}_{M_1,\varepsilon}\wedge\tau^{2,n}_{M_2,\varepsilon}}e^{-k\varepsilon\int^r_0U_{n,\varepsilon}(l)dl}\|u_{n,\varepsilon}(r)-v_{n,\varepsilon}(r)\|^2_H\right]^p\right)^{\frac{2}{p}}ds.
\end{align*}

Then Gronwall's inequality implies that 
\begin{equation}\label{eq3 in lemma7}\aligned
&\left(E\left[\sup_{0\leqslant t\leqslant \tau^{1,n}_{M_1,\varepsilon}\wedge\tau^{2,n}_{M_2,\varepsilon}}\|u_{n,\varepsilon}(t)-v_{n,\varepsilon}(t)\|^2_H\right]^p\right)^{\frac{2}{p}}\\
\leqslant&\left(E\left[\sup_{0\leqslant t\leqslant \tau^{1,n}_{M_1,\varepsilon}\wedge\tau^{2,n}_{M_2,\varepsilon}}(e^{-k\varepsilon\int^t_0U_{n,\varepsilon}(s)ds}\|u_{n,\varepsilon}(t)-v_{n,\varepsilon}(t)\|^2_H)^pe^{kp\varepsilon\int^{\tau^{1,n}_{M_1,\varepsilon}\wedge\tau^{2,n}_{M_2,\varepsilon}}_0U_{n,\varepsilon}(s)ds}\right]\right)^{\frac{2}{p}}\\
\leqslant& e^{C(\varepsilon+M_2)}CM_1^2\varepsilon^2 e^{C(\varepsilon^2+\varepsilon p)}.
\endaligned
\end{equation}

By Lemmas \ref{lemma3} and  \ref{lemma4}, we know that for any $R>0$, there exists $M_1$ such that 
$$\sup_{0<\varepsilon\leqslant 1} \varepsilon \log P\left(F_{u_{n,\varepsilon}}(T)>M_1\right)\leqslant -R,$$
$$\sup_{0<\varepsilon\leqslant 1} \varepsilon \log P\left(\sup_{0\leqslant t\leqslant T}\|v_{n,\varepsilon}(t)\|^2_{V}>M_1\right)\leqslant -R.$$

For such a constant $M_1$, by Lemma \ref{lemma3}, there exists $M_2$ such that
$$\sup_{0<\varepsilon\leqslant 1} \varepsilon \log P\left(G_{u_{n,\varepsilon}}(\tau^n_{M_1,\varepsilon})>M_2\right)\leqslant -R.$$

Then for such $M_1,M_2$, let $p=\frac{2}{\varepsilon}$ in (\ref{eq3 in lemma7}), we obtain
\begin{align*}
&\varepsilon \log P\left(\sup_{0\leqslant t\leqslant \tau^{1,n}_{M_1,\varepsilon}\wedge\tau^{2,n}_{M_2,\varepsilon}}\|u_{n,\varepsilon}(t)-v_{n,\varepsilon}(t)\|^2_H>\delta\right)\\
\leqslant&\log \left(E\left[\sup_{0\leqslant t\leqslant \tau^{1,n}_{M_1,\varepsilon}\wedge\tau^{2,n}_{M_2,\varepsilon}}\|u_{n,\varepsilon}(t)-v_{n,\varepsilon}(t)\|^2_H\right]^p\right)^{\frac{2}{p}}-\log\delta^2\\
\leqslant& C(\varepsilon+M_2)+\log[CM^2_1\varepsilon^2]+C(\varepsilon^2+1)-\log\delta^2\\
\rightarrow& -\infty \text{ }\text{ as }\varepsilon\rightarrow 0,
\end{align*}
where we used Chebyshev's inequality in the first inequality. Thus there exists a $\varepsilon_0\in(0,1)$ such that for any $\varepsilon\in(0,\varepsilon_0)$, 
$$P\left(\sup_{0\leqslant t\leqslant \tau^{1,n}_{M_1,\varepsilon}\wedge\tau^{2,n}_{M_2,\varepsilon}}\|u_{n,\varepsilon}(t)-v_{n,\varepsilon}(t)\|^2_H>\delta\right)\leqslant e^{-\frac{R}{\varepsilon}}.$$

Putting the above estimate together, by (\ref{eq in lemma7}) we see that for $\varepsilon\in(0,\varepsilon_0)$
$$P\left(\sup_{0\leqslant t\leqslant T}\|u_{n,\varepsilon}(t)-v_{n,\varepsilon}(t)\|^2_H>\delta\right)\leqslant 4e^{-\frac{R}{\varepsilon}}.$$

Since $R$ is arbitrary, we finish the proof.

\end{proof}

\begin{proof}[Proof of Theorem \ref{main}]

By Lemma \ref{LDP for v},  $v_\varepsilon$ satisfies a large deviation principle with the rate function $I^{u_0}$. Our task remain is to show that $u_\varepsilon$ and $v_\varepsilon$ are exponentially equivalent, then the result follows from Lemma \ref{EXEQ}.

By Lemmas \ref{lemma5} and \ref{lemma6}, for any $R>0$, there exists a $N_0$ such that for any $\varepsilon\in(0,1]$,
$$P\left(\sup_{0\leqslant t\leqslant T}\|u_{\varepsilon}(t)-u_{N_0,\varepsilon}(t)\|_H^2>\frac{\delta}{3}\right)\leqslant e^{-\frac{R}{\varepsilon}},$$
and
$$P\left(\sup_{0\leqslant t\leqslant T}\|v_{\varepsilon}(t)-v_{N_0,\varepsilon}(t)\|_H^2>\frac{\delta}{3}\right)\leqslant e^{-\frac{R}{\varepsilon}}.$$

Then by Lemma \ref{lemma7}, for such $N_0$, there exists a $\varepsilon_0$ such that for any $\varepsilon\in(0,\varepsilon_0)$,
$$P\left(\sup_{0\leqslant t\leqslant T}\|u_{N_0, \varepsilon}(t)-v_{N_0,\varepsilon}(t)\|_H^2>\frac{\delta}{3}\right)\leqslant e^{-\frac{R}{\varepsilon}}.$$

Therefore we deduce that for $\varepsilon\in(0,\varepsilon_0)$
$$P\left(\sup_{0\leqslant t\leqslant T}\|u_{\varepsilon}(t)-v_{\varepsilon}(t)\|_H^2>\delta\right)\leqslant 3e^{-\frac{R}{\varepsilon}}.$$

Since $R$ is arbitrary, we finish the proof.
\end{proof}

\appendix
\section{Appendix}

We now present several  lemmas from \cite{LZZ18}.  It follows from Minkowski  inequality that
\begin{lemma}\label{anisotropic Lp spaces}
For $1\leqslant q\leqslant p\leqslant \infty$, we have
$$\|u\|_{L^p_h(L^q_v)}\leqslant \|u\|_{L^q_v(L^p_h)},$$
$$\|u\|_{L^p_v(L^q_h)}\leqslant \|u\|_{L^q_h(L^p_v)}.$$
\end{lemma}

\begin{lemma}[{\cite[Lemma 3.4]{LZZ18}}]\label{estimate for anisotropic spaces}
Let $u$ be a smooth function from $\mathbb{T}^2$ to $\mathbb{R}$, we have
$$\|u\|^2_{L^2_v(L^\infty_h)}\leqslant C(\|u\|_{L^2}\|\partial_1 u\|_{L^2}+\|u\|_{L^2}^2),$$
$$\|u\|^2_{L^2_h(L^\infty_v)}\leqslant C(\|u\|_{L^2}\|\partial_2 u\|_{L^2}+\|u\|_{L^2}^2).$$
\end{lemma}
The following anisotropic estimate is from the proof of  \cite[Theorem 3.1]{LZZ18}:
\begin{lemma}\label{anisotropic estimate for b}
For smooth functions $u,v$ from $\mathbb{T}^2$ to $\mathbb{R}$ with $u$ satisfies the divergence free condition, we have 
\begin{align*}
|b(u,v,u)|\leqslant a\|\partial_1u\|^2_{L^2}+C\|u\|^2_{L^2}\Big{(}&\|\partial_1v\|^\frac{2}{3}_{L^2}\|\partial_1\partial_2 v\|^\frac{2}{3}_{L^2}+\|\partial_2 v\|^\frac{2}{3}_{L^2}\|\partial_1\partial_2 v\|^\frac{2}{3}_{L^2}\\
&+\|\partial_1v\|^2_{L^2}+\|\partial_1v\|_{L^2}+\|\partial_2v\|^2_{L^2}+\|\partial_2v\|_{L^2}\\
&+\|\partial_1v\|^\frac{1}{2}_{L^2}\|\partial_1\partial_2 v\|^\frac{1}{2}_{L^2}+\|\partial_2 v\|^\frac{1}{2}_{L^2}\|\partial_1\partial_2 v\|^\frac{1}{2}_{L^2}\Big{)},
\end{align*}
where $a>0$ is a constant small enough.

In particular, we have
$$|b(u,v,u)|\leqslant a\|\partial_1u\|^2_{L^2}+C\|u\|^2_{L^2}(1+\|v\|^2_{{H}^{1,1}}).$$
\end{lemma} 
\begin{proof}
We have
\begin{align*}
|b(u,v,u)|&=|\langle u^1\partial_1v+u^2\partial_2v, u\rangle|\\
&\leqslant (\|u^1\|_{L^\infty_h(L^2_v)}\|\partial_1 v\|_{L^2_h(L^\infty_v)}+\|u^2\|_{L^2_h(L^\infty_v)}\|\partial_2 v\|_{L^\infty_h(L^2_v)})\|u\|_{L^2},
\end{align*}
where $u=(u^1, u^2)$. Now we show the calculation of two terms in the right hand side separately.

 For the first term,  by Lemmas \ref{anisotropic Lp spaces} and \ref{estimate for anisotropic spaces}, we have
\begin{align*}
&\|u^1\|_{L^\infty_h(L^2_v)}\|\partial_1 v\|_{L^2_h(L^\infty_v)}\|u\|_{L^2}\\
\leqslant&C\|u\|_{L^2}\left(\|u^1\|_{L^2}\|\partial_1u^1\|_{L^2}+\|u^1\|^2_{L^2}\right)^\frac{1}{2}\left(\|\partial_1v\|_{L^2}\|\partial_1\partial_2 v\|_{L^2}+\|\partial_1 v\|^2_{L^2}\right)^\frac{1}{2}\\
\leqslant& C\|u\|_{L^2} \left(\|u^1\|_{L^2}\|\partial_1u^1\|_{L^2}\|\partial_1v\|_{L^2}\|\partial_1\partial_2 v\|_{L^2}\right)^\frac{1}{2}+C\|u\|_{L^2}\|u^1\|_{L^2}\|\partial_1v\|_{L^2}\\
&+C\|u\|_{L^2}(\|u^1\|_{L^2}+\|\partial_1u^1\|_{L^2})\|\partial_1v\|_{L^2}+C\|u\|_{L^2}\|u^1\|_{L^2}\|\partial_1v\|^{\frac{1}{2}}_{L^2}\|\partial_1\partial_2 v\|^{\frac{1}{2}}_{L^2}.
\end{align*}

Then Young's inequality implies that 
\begin{align*}
&C\|u\|_{L^2} \left(\|u^1\|_{L^2}\|\partial_1u^1\|_{L^2}\|\partial_1v\|_{L^2}\|\partial_1\partial_2 v\|_{L^2}\right)^\frac{1}{2}\\
\leqslant &\frac{a}{4}\|\partial_1 u\|^2_{L^2}+C\|\partial_1v\|^{\frac{2}{3}}_{L^2}\|\partial_1\partial_2v\|^\frac{2}{3}_{L^2}\|u\|^2_{L^2},
\end{align*}
and
\begin{align*}
&C\|u\|_{L^2}\|\partial_1u^1\|_{L^2}\|\partial_1v\|_{L^2}\leqslant \frac{a}{4}\|\partial_1u\|^2_{L^2}+C\|\partial_1v\|^2_{L^2}\|u\|^2_{L^2}.
\end{align*}

Thus we have
\begin{align*}
&\|u^1\|_{L^\infty_h(L^2_v)}\|\partial_1 v\|_{L^2_h(L^\infty_v)}\|u\|_{L^2}\\
\leqslant&\frac{a}{2}\|\partial_1u\|^2_{L^2}+C\|u\|^2_{L^2}\Big{(}\|\partial_1v\|^{\frac{2}{3}}_{L^2}\|\partial_1\partial_2v\|^\frac{2}{3}_{L^2}+\|\partial_1v\|^2_{L^2}+\|\partial_1v\|_{L^2}+\|\partial_1v\|^{\frac{1}{2}}_{L^2}\|\partial_1\partial_2 v\|^{\frac{1}{2}}_{L^2}\Big{)}.
\end{align*}

Do the same calculation for the second term and combine the divergence free condition $\partial_2u^2=-\partial_1u^1$, we have
\begin{align*}
&\|u^2\|_{L^2_h(L^\infty_v)}\|\partial_2 v\|_{L^\infty_h(L^2_v)}\|u\|_{L^2}\\
\leqslant&\frac{a}{2}\|\partial_1u\|^2_{L^2}+C\|u\|^2_{L^2}\Big{(}\|\partial_2v\|^{\frac{2}{3}}_{L^2}\|\partial_1\partial_2v\|^\frac{2}{3}_{L^2}+\|\partial_2v\|^2_{L^2}+\|\partial_2v\|_{L^2}+\|\partial_2v\|^{\frac{1}{2}}_{L^2}\|\partial_1\partial_2 v\|^{\frac{1}{2}}_{L^2}\Big{)},
\end{align*}
which implies the first inequality.

The second inequality holds from the first one and Young's Inequality.

\end{proof}

Similar to the proof of Lemma \ref{anisotropic estimate for b}, by Lemmas \ref{anisotropic Lp spaces} and \ref{estimate for anisotropic spaces}, we also have
\begin{lemma}\label{b(u,v,w)}
For smooth functions $u,v,w$ form $\mathbb{T}^2$ to $\mathbb{R}^2$ with divergence free condition, we have
$$|b(u,v,w)|\leqslant C\|u\|_{H^1}\|v\|_{H^{1,1}}\|w\|_{L^2}.$$
\end{lemma}
\begin{proof}
\begin{align*}
&|b(u, v, w)|\\
\leqslant &(\|u^1\|_{L^\infty_h(L^2_v)}\|\partial_1v\|_{L^2_h(L^\infty_v)}+\|u^2\|_{L^2_h(L^\infty_v)}\|\partial_2 v\|_{L^\infty_h(L^2_v)})\|w\|_{L^2}\\
\leqslant&C\Big{(}(\|u^1\|_{L^2}\|\partial_1u^1\|_{L^2}+\|u^1\|^2_{L^2})^\frac{1}{2}(\|\partial_1v\|_{L^2}\|\partial_1\partial_2v\|_{L^2}+\|\partial_1v\|^2_{L^2})^\frac{1}{2}\\
&+(\|u^2\|_{L^2}\|\partial_2u^2\|_{L^2}+\|u^2\|_{L^2}^2)^\frac{1}{2}(\|\partial_2v\|_{L^2}\|\partial_1\partial_2v\|_{L^2}+\|\partial_2v\|_{L^2}^2)^\frac{1}{2}\Big{)}\|w\|_{L^2}\\
\leqslant&C\|u\|_{H^1}\|v\|_{H^{1,1}}\|w\|_{L^2}.
\end{align*}
\end{proof}

The next lemma is from the proof of \cite[Lemma 3.5]{LZZ18}, which plays an important role in $H^{0,1}$-estimate.

\begin{lemma}\label{estimate for b with partial_2}
For smooth function $u$ form $\mathbb{T}^2$ to $\mathbb{R}^2$ with divergence free condition, we have
$$|\langle \partial_2u, \partial_2(u\cdot \nabla u)\rangle|\leqslant a\|\partial_1\partial_2 u\|^2_{L^2}+C(1+\|\partial_1 u\|^2_{L^2})\|\partial_2 u\|^2_{L^2},$$
where $a>0$ is a constant small enough.
\end{lemma}

\begin{proof}
We have
$$\langle \partial_2u, \partial_2(u\cdot \nabla u)\rangle= \langle \partial_2u^1, \partial_2(u\cdot \nabla u^1)\rangle+\langle \partial_2u^2, \partial_2(u\cdot \nabla u^2)\rangle,$$
where $u=(u^1,u^2)$.

For the first term on the right hand side, we have
\begin{align*}
\langle \partial_2u^1, \partial_2(u\cdot \nabla u^1)\rangle=&\langle \partial_2u^1, \partial_2(u^1\partial_1u^1+u^2\partial_2u^1)\rangle\\
=&\langle \partial_2u^1, \partial_2u^1\partial_1u^1\rangle+\langle \partial_2u^1, u^1\partial_2\partial_1u^1\rangle\\
&+\langle \partial_2u^1, \partial_2u^2\partial_2u^1\rangle+\langle \partial_2u^1, u^2\partial_2^2u^1\rangle\\
=& \langle \partial_2u^1, u^1\partial_2\partial_1u^1\rangle+\langle \partial_2u^1, u^2\partial_2^2u^1\rangle\\
=&\langle \partial_2u^1, u\cdot\nabla\partial_2u^1\rangle\\
=&-\frac{1}{2}\int \text{div }u|\partial_2u^1|^2dx\\
=&0,
\end{align*}
where we use the fact  $\text{div }u=0$ in the third and sixth equality.

Similarly, for the second term, we have
\begin{align*}
\langle \partial_2u^2, \partial_2(u\cdot \nabla u^2)\rangle=&\langle \partial_2u^2, \partial_2u^1\partial_1u^2\rangle+\langle \partial_2u^2, u^1\partial_2\partial_1u^2\rangle\\
&+\langle \partial_2u^2, \partial_2u^2\partial_2u^2\rangle+\langle \partial_2u^2, u^2\partial_2^2u^2\rangle\\
=& \langle \partial_2u^2, \partial_2u^1\partial_1u^2\rangle +\frac{1}{2}\int u^1\partial_1(\partial_2u^2)^2dx\\
&+\langle \partial_2u^2, \partial_2u^2\partial_2u^2\rangle+\frac{1}{2}\int  u^2\partial_2(\partial_2u^2)^2dx\\
=& \langle \partial_2u^2, \partial_2u^1\partial_1u^2\rangle +\langle \partial_2u^2, \partial_2u^2\partial_2u^2\rangle\\
&-\frac{1}{2}\langle \partial_2u^2, \partial_1u^1\partial_2u^2\rangle-\frac{1}{2}\langle \partial_2u^2, \partial_2 u^2\partial_2u^2\rangle\\
=& \langle \partial_2u^2, \partial_2u^1\partial_1u^2\rangle +\langle \partial_2u^2, \partial_2u^2\partial_2u^2\rangle,
\end{align*}
where we use $\text{div } u=0$ in the last equality.

Then by Lemma \ref{estimate for anisotropic spaces} we have
\begin{align*}
&|\langle \partial_2u, \partial_2(u\cdot \nabla u)\rangle|\\
=&|\langle \partial_2u^2, \partial_2u^1\partial_1u^2\rangle +\langle \partial_2u^2, \partial_2u^2\partial_2u^2\rangle|\\
\leqslant &\left(\|\partial_2u^1\|_{L^{\infty}_h(L^2_v)}\|\partial_1u^2\|_{L^2_h(L^{\infty}_v)}+\|\partial_1u^1\|_{L^2_h(L^{\infty}_v)}\|\partial_2u^2\|_{L^{\infty}_h(L^2_v)}\right)\|\partial_2u^2\|_{L^2}\\
\leqslant & C\left(\|\partial_2u\|_{L^2}+\|\partial_2u\|^{\frac{1}{2}}_{L^2}\|\partial_1\partial_2u\|^{\frac{1}{2}}_{L^2}\right)\left(\|\partial_1u\|_{L^2}+\|\partial_1u\|^{\frac{1}{2}}_{L^2}\|\partial_1\partial_2u\|^{\frac{1}{2}}_{L^2}\right)\|\partial_2u^2\|_{L^2}\\
\leqslant &C\|\partial_1u\|_{L^2}\|\partial_2u\|^2_{L^2}+C\|\partial_1\partial_2u\|_{L^2}\|\partial_1u\|_{L^2}\|\partial_2u\|_{L^2}\\
&+C\|\partial_1\partial_2u\|^{\frac{1}{2}}_{L^2}\left(\|\partial_1u\|_{L^2}\|\partial_2u\|^{\frac{1}{2}}_{L^2}+\|\partial_2u\|_{L^2}\|\partial_1u\|^{\frac{1}{2}}_{L^2}\right)\|\partial_2u^2\|_{L^2},
\end{align*}
where we use the following inequality in the last inequality:
\begin{align*}
&\|\partial_2 u\|^{\frac{1}{2}}_{L^2}\|\partial_1\partial_2u\|_{L^2}\|\partial_1u\|^{\frac{1}{2}}_{L^2}\|\partial_2u^2\|_{L^2}\\
=&\|\partial_2 u\|^{\frac{1}{2}}_{L^2}\|\partial_1\partial_2u\|_{L^2}\|\partial_1u\|^{\frac{1}{2}}_{L^2}\|\partial_1u^1\|^{\frac{1}{2}}_{L^2}\|\partial_2u^2\|^{\frac{1}{2}}_{L^2}\\
\leqslant&\|\partial_1\partial_2u\|_{L^2}\|\partial_1u\|_{L^2}\|\partial_2u\|_{L^2},
\end{align*}
where we use $\text{div } u=0$ in the first equality.

By Young's inequality, we have
$$C\|\partial_1\partial_2u\|_{L^2}\|\partial_1u\|_{L^2}\|\partial_2u\|_{L^2}\leqslant \frac{a}{2}\|\partial_1\partial_2u\|^2_{L^2}+C\|\partial_1u\|^2_{L^2}\|\partial_2u\|^2_{L^2},$$
and
\begin{align*}
&C\|\partial_1\partial_2u\|^{\frac{1}{2}}_{L^2}\left(\|\partial_1u\|_{L^2}\|\partial_2u\|^{\frac{1}{2}}_{L^2}+\|\partial_2u\|_{L^2}\|\partial_1u\|^{\frac{1}{2}}_{L^2}\right)\|\partial_2u^2\|_{L^2}\\
\leqslant& \frac{a}{2}\|\partial_1\partial_2u\|^{2}_{L^2}+C\left(\|\partial_1u\|^{\frac{4}{3}}_{L^2}\|\partial_2u\|^{\frac{2}{3}}_{L^2}+\|\partial_2u\|^{\frac{4}{3}}_{L^2}\|\partial_1u\|^{\frac{2}{3}}_{L^2}\right)\|\partial_2u^2\|^{\frac{4}{3}}_{L^2}\\
\leqslant & \frac{a}{2}\|\partial_1\partial_2u\|^{2}_{L^2}+C\|\partial_1u\|^{\frac{4}{3}}_{L^2}\|\partial_2u\|^{2}_{L^2}+C\|\partial_2u\|^{\frac{4}{3}}_{L^2}\|\partial_1u\|^{\frac{2}{3}}_{L^2}\|\partial_1u^1\|^{\frac{2}{3}}_{L^2}\|\partial_2u^2\|^{\frac{2}{3}}_{L^2}\\
\leqslant& \frac{a}{2}\|\partial_1\partial_2u\|^{2}_{L^2}+C(1+\|\partial_1u\|^{2}_{L^2})\|\partial_2u\|^{2}_{L^2},
\end{align*}
where we use $\text{div }u=0$ in the second inequality.

Thus we deduce that 
$$|\langle \partial_2u, \partial_2(u\cdot \nabla u)\rangle|\leqslant a\|\partial_1\partial_2 u\|^2_{L^2}+C(1+\|\partial_1 u\|^2_{L^2})\|\partial_2 u\|^2_{L^2}.$$

\end{proof}

\bibliography{Ref0}{}  
\bibliographystyle{alpha}

\end{document}